\documentclass{article}
\usepackage{amsmath}
\usepackage{amsfonts}
\usepackage{amsthm}
\usepackage{amssymb}
\usepackage{latexsym}
\usepackage[pdftex]{graphicx}
\usepackage[top=30truemm,bottom=30truemm,left=30truemm,right=30truemm]{geometry}
\theoremstyle{definition}
\newtheorem{dfn}{Definition}[section]
\newtheorem{thm}[dfn]{Theorem}
\newtheorem{lem}[dfn]{Lemma}

\newtheorem{prop}[dfn]{Proposition}
\newtheorem{ex}[dfn]{Example}
\newtheorem{remk}[dfn]{Remark}

\newtheorem{mthmA}{Theorem A\!\!}[]
\newtheorem{mthmB}{Theorem B\!\!}[]
\newtheorem{mthmC}{Theorem C\!\!}[]

\begin{document}
\title{Decorated enhanced Teichm\"{u}ller spaces}
\author{Katsuhiro Miguchi}
\date{}
\maketitle
\begin{abstract}
\noindent
In this paper, we introduce a new variation of the Teichm\"{u}ller space, namely the deformation space of hyperbolic structures on a surface with both enhancement and decoration. 
We construct the parameterization of this deformation space, which is a common generalization of the shear coordinates and the $\lambda$-length coordinates. 
Furthermore, we introduce the lamination space corresponding to this deformation space, and show the compatibility of the shear coordinates and the $\lambda$-length coordinates. 
\end{abstract}

\section{Introduction}
\noindent
The enhanced Teichm\"{u}ller spaces are considered when we transform the triangulated hyperbolic surfaces by shearing along the edges. 
The decorated Teichm\"{u}ller spaces are introduced by R. C. Penner \cite{Penner2}. 
V. V. Fock and A. B. Goncharov improved the theory of these moduli spaces \cite{Fock-Goncharov}. 
There are the canonical correspondence of these Teichm\"{u}ller spaces and the corresponding lamination spaces. 
The aim of this paper is to construct a common generalization of the enhanced and decorated Teichm\"{u}ller spaces, and establish a correspondence between the generalized Teichm\"{u}ller space and a corresponding lamination space. 
\par
Let $S$ be an oriented connected compact surface of genus $g\ge 0$ with $c\ge 0$ boundary components. 
Remove $p\ge 0$ points (puncture vertices) from the interior of $S$, and $s\ge 0$ points (spike vertices) from the boundaries of $S$. 
Throughout we assume that all boundary components have at least one spike vertex and that $4-4g-2p-2c-s<0$. 
Then $S$ admits a hyperbolic structure with crown-shaped boundaries. 
\par
An {\it enhancement} is an assignment of signs to the closed geodesic boundaries (see \S 2.2). 
We define the {\it decoration} as a collection of horocycles centered at the cusps and the spikes (see \S 2.3) and moreover equidistant curves at the closed geodesic boundaries (see \S 3.1). 
Then, we introduce the decorated enhanced Teichm\"{u}ller space $\mathcal{T}^{ax}(S)$ of $S$ as the deformation space of all marked hyperbolic surfaces homeomorphic to $S$ with both enhancement and decoration. 
\par
Let $\Gamma$ be a triangulation obtained by connecting the $p+s$ vertices on $S$. 
Let $V$ denote the set of the vertices, $V^i$ denote the set of the puncture vertices, $E$ denote the set of the edges, and $E^i$ denote the set of the edges which are in the interior of $S$. 
A {\it shear parameter} of an interior edge is the signed length parameter for gluing the ideal triangles along the edge. 
It is well known that the enhanced Teichm\"{u}ller space $\mathcal{T}^{x}(S)$ of $S$ is parameterized by the shear coordinates \cite{Fock-Goncharov} (see Proposition \ref{Prop shear coord} in \S 2.2): 
$$\varphi_x : \mathcal{T}^x(S)\xrightarrow{\mathrm{diffeo. }}\mathbb{R}^{E^i}. $$
In this paper, we introduce decoration parameters for boundary components, and give a parameterization of the decorated enhanced Teichm\"{u}ller space of $S$: 
\begin{mthmA}
(Theorem \ref{Thm shear deco coord} in \S 3.2) Let 
$$\psi_x : \mathcal{T}^{ax}(S)\rightarrow\mathbb{R}^{E^i}\times\mathbb{R}^V$$
be the map giving the shear parameters and the decoration parameters. 
Then, $\psi_x$ is a homeomorphism. 
\end{mthmA}
\par
A {\it $\lambda$-length} of an edge is the signed length between the horocycles corresponding to its endpoints. 
It is also well known that the decorated Teichm\"{u}ller space $\mathcal{T}^a(S)$ of $S$ is parameterized by the $\lambda$-length coordinates \cite{Fock-Goncharov} (see Proposition \ref{Prop l-length coord} in \S 2.3): 
$$\varphi_a : \mathcal{T}^a(S)\xrightarrow{\mathrm{diffeo. }}\mathbb{R}^{E}. $$
In this paper, we calculate the relation given by a canonical map between the shear parameters and the $\lambda$-lengths: 
$$\psi : \mathbb{R}^{E^i}\times\mathbb{R}^V\rightarrow\mathbb{R}^E\times\mathbb{R}^{V^i}, $$
and analyze the generalized $\lambda$-length coordinates of the decorated enhanced Teichm\"{u}ller space. 
If $S$ has a spike, there are triangulations whose puncture vertices are all 1-valent. 
For such triangulations, we give a parameterization of the decorated enhanced Teichm\"{u}ller space of $S$: 
\begin{mthmB}
(Theorem \ref{Thm l-b-len coord} in \S 4.2) For a triangulation whose puncture vertices are 1-valent, let 
$$\psi_a : \mathcal{T}^{ax}(S)\rightarrow\mathbb{R}^E\times\mathbb{R}^{V^i}$$
be the map giving the $\lambda$-lengths and the sighed boundary lengths. 
Then, $\psi_a$ is a homeomorphism. 
\end{mthmB}
\par
Next, we consider certain lamination spaces corresponding to the shear parameters and the $\lambda$-lengths. 
There are a homeomorphism between the space $T^x(S, \mathbb{R})$ of real $\mathcal{X}$-laminations and $\mathbb{R}^{E^i}$ \cite{Fock-Goncharov} (see Proposition \ref{Prop coord of X-lami} in \S 5.1) 
$$\Phi_x : T^x(S, \mathbb{R})\xrightarrow{\mathrm{homeo. }}\mathbb{R}^{E^i}, $$
and a homeomorphism between the space $T^a(S, \mathbb{R})$ of real $\mathcal{A}$-laminations and $\mathbb{R}^E$ \cite{Fock-Goncharov} (see Proposition \ref{Prop coord of A-lami} in \S 5.1) 
$$\Phi_a : T^a(S, \mathbb{R})\xrightarrow{\mathrm{homeo. }}\mathbb{R}^E. $$
In this paper, we introduce the space $T^{ax}(S, \mathbb{R})$ of real $\mathcal{AX}$-laminations and give its parametrization by the homeomorphism 
$$\Psi_x : T^{ax}(S, \mathbb{R})\xrightarrow{\mathrm{homeo. }}\mathbb{R}^{E^i}\times\mathbb{R}^V. $$
Moreover, we show the compatibility of $\Phi_x$ and $\Phi_a$: 
\begin{mthmC}
(Theorem \ref{Thm compati of para} in \S 5.3) Let
$$\Psi_a : T^{ax}(S, \mathbb{R})\rightarrow\mathbb{R}^E\times\mathbb{R}^{V^i}$$
be the composition of $\Psi_x$ and $\psi$. 
Then the equations 
$$\Psi_x\circ I_x=\iota_x\circ\Phi_x\mathrm{\ and\ }\Psi_a\circ I_a=\iota_a\circ\Phi_a $$
hold, where $\iota_x$ is the inclusion map from $\mathbb{R}^{E^i}$ to $\mathbb{R}^{E^i}\times\mathbb{R}^V$, $\iota_a$ is the inclusion map from $\mathbb{R}^E$ to $\mathbb{R}^E\times\mathbb{R}^{V^i}$, $I_x$ is the inclusion map from $T^x(S, \mathbb{R})$ to $T^{ax}(S, \mathbb{R})$, and $I_a$ is the inclusion map from $T^a(S, \mathbb{R})$ to $T^{ax}(S, \mathbb{R})$ as in Figure \ref{Fig diag3}. 
\end{mthmC}
\begin{figure}[htbp]
\begin{center}
\includegraphics[width=12cm]{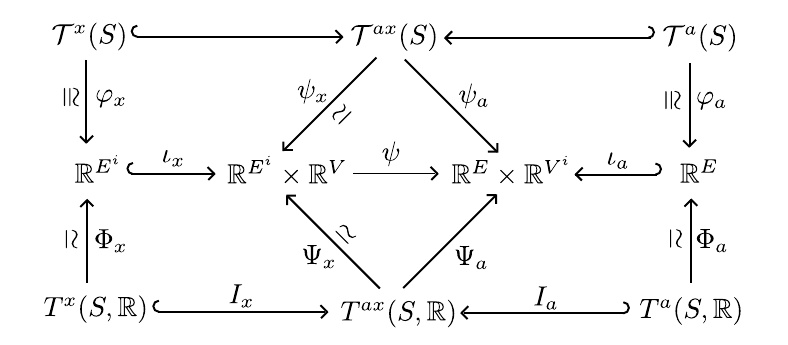}
\caption{The commutative diagram describing that $\mathcal{T}^{ax}(S)$ is a common generalization of $\mathcal{T}^x(S)$ and $\mathcal{T}^a(S)$, and the compatibility of $\Phi_x$ and $\Phi_a$. 
The symbol $\simeq$ means homeomorphisms and the symbol $\cong$ means diffeomorphisms. }
\label{Fig diag3}
\end{center}
\end{figure}
Theorem A and Theorem B imply the upper side of the diagram in Figure \ref{Fig diag3}, and Theorem C implies the lower side of the diagram in Figure \ref{Fig diag3}. 
There is not obvious choice of the origin of the decoration parameter. 
We carefully pick an origin so that Theorem C holds. 
Unfortunately, with this choice of the origin, the map $\psi$ is only a piecewise diffeomorphism. 
For another choice of the origin, the map $\psi_x$ in Theorem A and the map $\psi_a$ in Theorem B are diffeomorphisms. 
\par
The organization of this paper is as follows. 
In \S 2, we review the definitions of enhanced and decorated Teichm\"{u}ller spaces. 
In \S 3.1, we introduce the decorated enhanced Teichm\"{u}ller space. 
In \S 3.2, we parametrize the decorations and prove Theorem A. 
In \S 3.3, we analyze the properties of decoration parameters. 
In \S 4.1, we give some calculations associated with generalized shear coordinates. 
In \S 4.2, we prove Theorem B. 
In \S 4.3, we note that the assumption in Theorem B is a sufficient condition. 
In \S 5.1, we review the definitions of the spaces of laminations corresponding to generalized Teichm\"{u}ller spaces. 
In \S 5.2, we introduce the space of laminations corresponding to the decorated enhanced Teichm\"{u}ller space. 
In \S 5.3, we prove Theorem C. 

\section{Enhanced and decorated Teichm\"{u}ller spaces}
\noindent
In this chapter, we introduce the Teichm\"{u}ller space and its variations. 
These generalized Teichm\"{u}ller spaces have the coordinate spaces, which will be generalized to the coordinate space of the decorated enhanced Teichm\"{u}ller space. 

\subsection{Teichm\"{u}ller spaces}
\noindent
In this section, we introduce the terminologies and the symbols used in this paper. 
We will define the standard deformation space of the marked hyperbolic surfaces, which is called the {\it Teichm\"{u}ller space}. 
\par
Let $\hat{S}$ be an oriented connected compact surface of genus $g \ge 0$ with $c \ge 0$ boundary components. 
Take $p \ge 0$ interior marked points and $s \ge 0$ boundary marked points on $\hat{S}$. 
Let $V^i$ denote the set of the interior marked points, $V^b$ denote the set of the boundary marked points, and $V$ be the union of $V^i$ and $V^b$. 
The surface $S = \hat{S} \setminus V$ is the reference surface of the hyperbolic surfaces. 
A point of $V^i$ is called a {\it puncture}, a point of $V^b$ is called a {\it spike},  and a boundary component of $\hat{S}$ is called a {\it crown} of $S$. 
We assume that all crowns have at least one spike and that $4-4g-2p-2c-s<0$. 
Then $S$ admits a hyperbolic structure with crown-shaped boundaries. 
\par
\begin{figure}[htbp]
\begin{center}
\includegraphics[width=14cm]{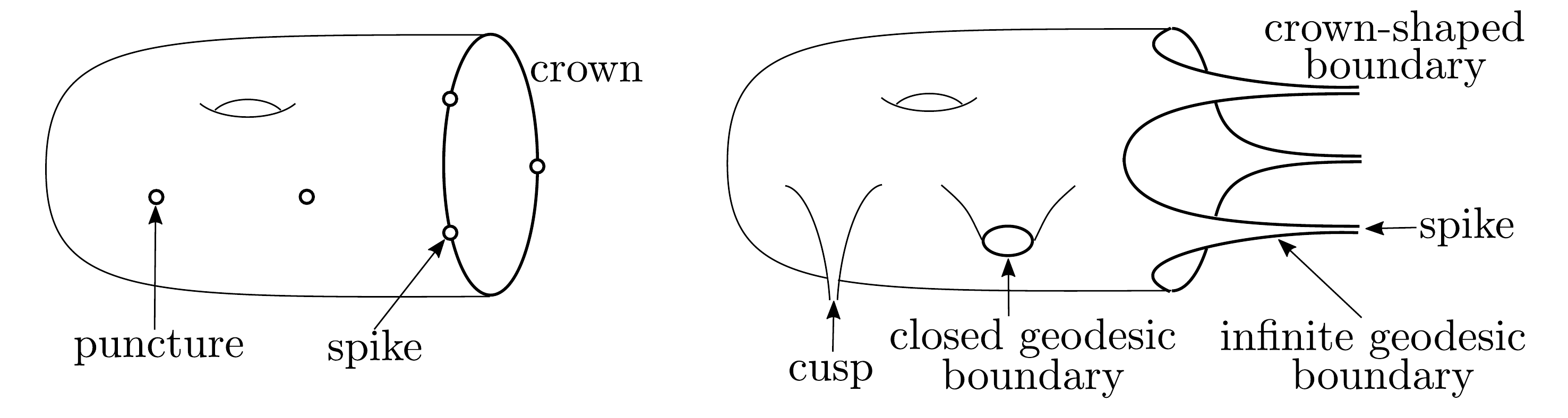}
\caption{A reference surface (left) and a hyperbolic surface (right). }
\label{Fig ref and hyp surf}
\end{center}
\end{figure}
Let $X$ be a complete 2-dimensional hyperbolic Riemannian manifold with totally geodesic boundaries. 
$X$ is called a {\it hyperbolic surface} of the reference surface $S$ if $X$ has a finite area and the subsurface $\mathring{X}$ of $X$ obtained by removing all closed geodesic boundaries is homeomorphic to $S$. 
A hyperbolic surface $X$ may have cusps, closed geodesic boundaries and crown-shaped boundaries that are composed of asymptotic infinite geodesic boundaries, and $X$ has no funnels (Figure \ref{Fig ref and hyp surf}). 
\par
\begin{dfn}\label{Def tri}(Triangulations)
A graph $\Gamma$ on $\hat{S}$ is called a {\it triangulation} of $S$ if the vertex set of $\Gamma$ is $V$ and each component of the complement of $\Gamma$ is a triangle bounded by (not necessary distinct) 3 edges and 3 vertices. 
\end{dfn}
A triangulation $\Gamma$ of $S$ makes an ideal triangulation of $X$. 
Map the edges of $\Gamma$ by a homeomorphism from $S$ to $\mathring{X}$, and isotope them to be geodesics, then an ideal triangulation is obtained. 
This ideal triangulation depends on the homeomorphism (that are the markings of $X$) and the deformations (that are the enhancements of $X$) around the closed geodesic boundaries (Figure \ref{Fig variety of ideal tri}). 
\par
\begin{figure}[htbp]
\begin{center}
\includegraphics[width=14cm]{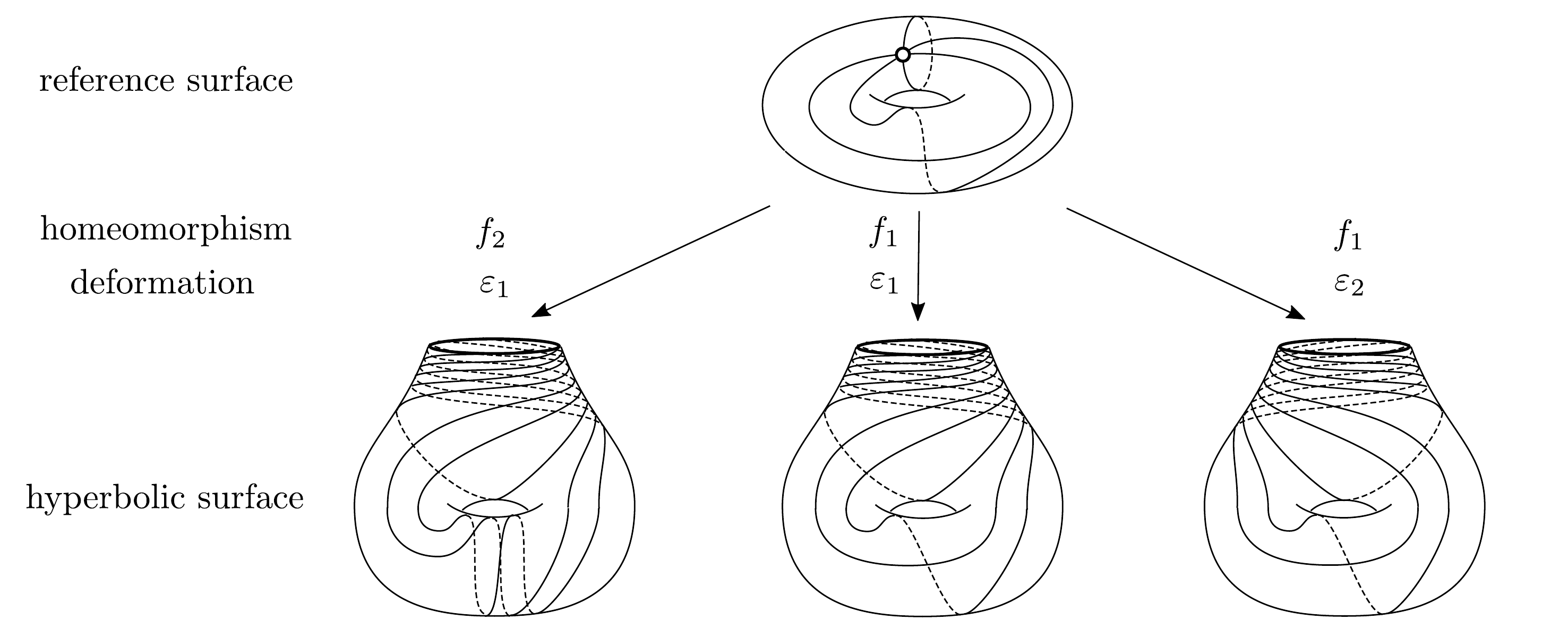}
\caption{Ideal triangulations depend on the marking $f_i:S\rightarrow\mathring{X}$ and the enhancement $\varepsilon_i$. }
\label{Fig variety of ideal tri}
\end{center}
\end{figure}
\begin{dfn}\label{Def mark}(Markings)
An orientation preserving homeomorphism $f:S\rightarrow\mathring{X}$ is called a {\it marking} of $X$ and a pair $(X, f)$ is called a {\it marked hyperbolic surface} of $S$. 
\end{dfn}
Marked hyperbolic surfaces $(X_1, f_1), (X_2, f_2)$ are {\it Teichm\"{u}ller equivalent} and denoted by $(X_1, f_1)\sim_T(X_2, f_2)$, if there is an orientation preserving isometry $\varphi$ from $X_1$ to $X_2$ such that the maps $\varphi\circ f_1, f_2$ are homotopic. 
Let define the deformation space of the marked hyperbolic structures on $S$. 
\begin{dfn}\label{Def Teich sp}(Teichm\"{u}ller spaces)
The set of Teichm\"{u}ller equivalence classes of marked hyperbolic surfaces of $S$ is called the {\it Teichm\"{u}ller space} of $S$, which is denoted by $\mathcal{T}(S)$. 
\end{dfn}
Sometimes the Teichm\"{u}ller space means the quotient space of the marked hyperbolic surfaces without closed geodesic boundaries.  
To avoid confusion, this is denoted by $\mathcal{T}_0(S)$. 

\subsection{Enhanced Teichm\"{u}ller spaces}
\noindent
In this section, we consider the additional data on the closed geodesic boundaries of the hyperbolic surfaces, which is called the {\it enhancement}. 
We will define the deformation space respecting the marking and the enhancement, and construct the shear coordinates of this space. 
\begin{dfn}\label{Def enha}(Enhancements)
Let $f$ be a marking of $X$ and $P(X)$ denote the collection of the cusps and the closed geodesic boundaries of $X$. 
A map $\varepsilon:P(X)\rightarrow\{0, \pm1\}$ is called an {\it enhancement} of $X$ if it maps cusps to 0 and maps closed geodesic boundaries to $\pm 1$. 
Then a triplet $(X, f, \varepsilon)$ is called an {\it enhanced marked hyperbolic surface} of $S$. 
\end{dfn}
Enhanced marked hyperbolic surfaces $(X_1, f_1, \varepsilon_1), (X_2, f_2, \varepsilon_2)$ are {\it enhanced equivalent} and denoted by $(X_1, f_1, \varepsilon_1)\sim_E(X_2, f_2, \varepsilon_2)$, if $(X_1, f_1)\sim_T(X_2, f_2)$ and $\varepsilon_2\circ\varphi_\ast=\varepsilon_1$ where the map $\varphi_\ast$ is induced from the isometry $\varphi$ preserving the marking. 
Since the enhancement is thought of as the signs of the lengths of the closed geodesic boundaries and cusps are regarded as the closed geodesic boundaries whose length are 0, the enhancement at the cusps can be ignored. 
Let us define the deformation space of the marked hyperbolic structures with the enhancement on $S$. 
\begin{dfn}\label{Def enha Teich sp}(Enhanced Teichm\"{u}ller spaces)
The set of enhanced equivalence classes of enhanced marked hyperbolic surfaces of $S$ is called the {\it enhanced Teichm\"{u}ller space} of $S$, which is denoted by 
$$\mathcal{T}^x(S)=\{\mathrm{Enhanced\ marked\ hyperbolic\ surfaces\ of\ }S\} /\sim_E. $$
\end{dfn}
For a map $\varepsilon_0:V^i\rightarrow\{0, \pm1\}$, let
$$\mathcal{T}^x_{\varepsilon_0}(S)=\{[X, f, \varepsilon]\in\mathcal{T}^x(S)\ |\ \varepsilon(C_v)=\varepsilon_0(v)\mathrm{\ or\ }0\mathrm{\ for\ all\ }v\in V^i\}, $$
where $C_v$ is the cusp or the closed geodesic boundary corresponding to $v$. 
$\mathcal{T}^x_{\varepsilon_0}(S)$ is identified with $\mathcal{T}_0(S)$ if $\varepsilon_0(V^i)=\{0\}$, and $\mathcal{T}^x_{\varepsilon_0}(S)$ is identified with $\mathcal{T}(S)$ if $\varepsilon_0(V^i)=\{\pm1\}$ by forgetting the enhancement. 
For maps $\varepsilon_1, \varepsilon_2:V^i\rightarrow\{\pm1\}$, $\mathcal{T}^x_{\varepsilon_1}(S)$ and $\mathcal{T}^x_{\varepsilon_2}(S)$ are glued along the subspace $\mathcal{T}^x_{\frac{1}{2}(\varepsilon_1+\varepsilon_2)}(S)$ in $\mathcal{T}^x(S)$. 
\begin{dfn}\label{Def shear para}(Shear parameters)
Consider an interior edge $e$ of a triangulation $\Gamma$. 
Let $T, T'$ denote the triangles which have $e$ as a boundary, $v_1, v_2, v_3$ denote the vertices of $T$, and $v_3, v_4, v_1$ denote the vertices of $T'$. 
Lift them to the upper half plane $\mathbb{H}^2$. 
The foot of the perpendicular from $\widetilde{v_2}$ to $\widetilde{e}$ is called the {\it base point} of $\widetilde{e}$ with respect to $\widetilde{T}$. 
The signed length along $\widetilde{e}$ from the base point $b$ of $\widetilde{e}$ with respect to $\widetilde{T}$ to the base point $b'$ of $\widetilde{e}$ with respect to $\widetilde{T'}$ is called the {\it shear parameter} of $e$, where the direction respects to the orientation of the boundary of $\widetilde{T}$ (Figure \ref{Fig shear para}). 
\end{dfn}
\begin{figure}[htbp]
\begin{center}
\includegraphics[width=14cm]{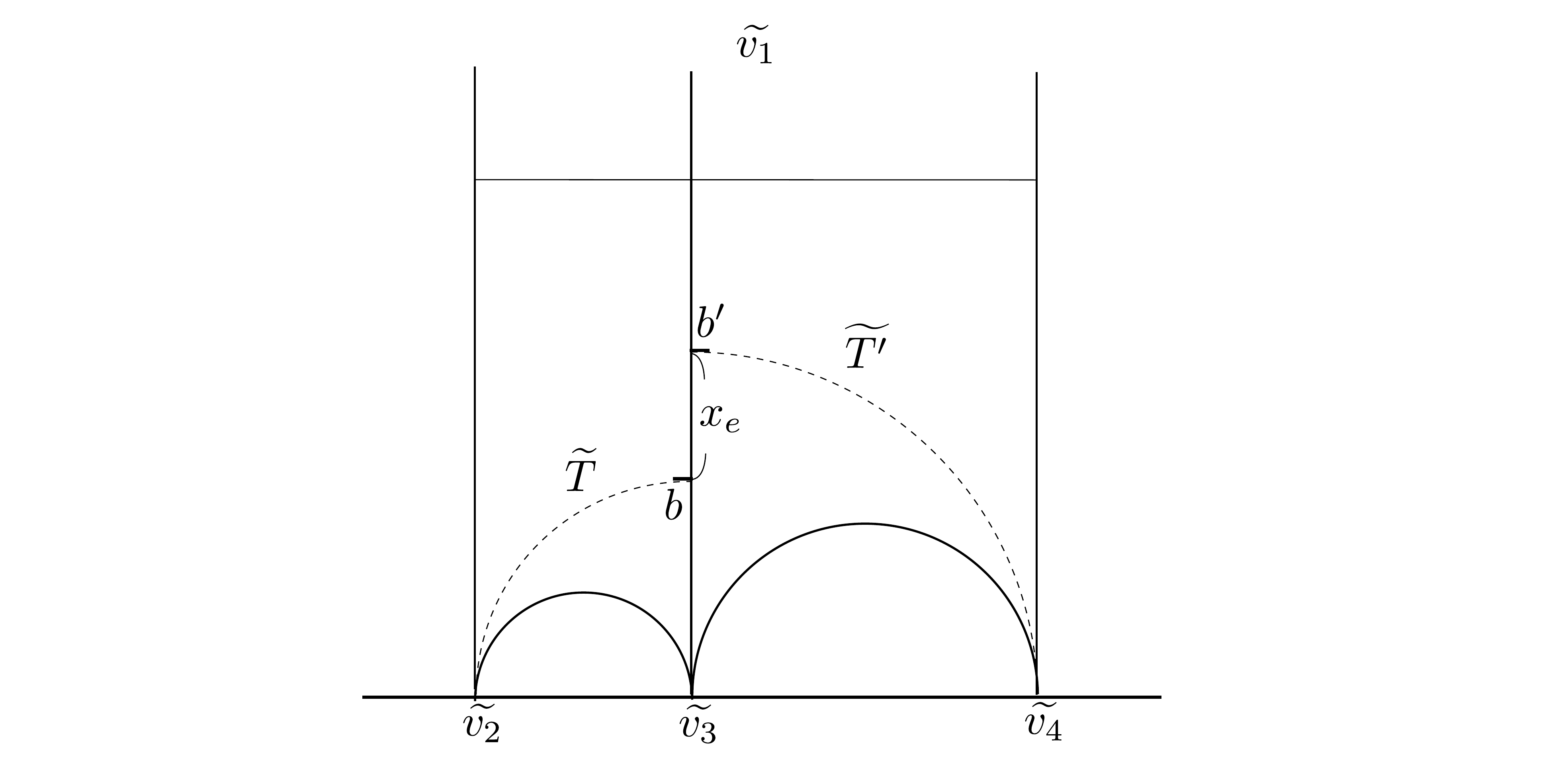}
\caption{The shear parameter $x_e$ of an edge $e$, and the base point $b$ of $\widetilde{e}$ with respect to $\widetilde{T}$. }
\label{Fig shear para}
\end{center}
\end{figure}
The enhancement of a puncture vertex $v$ is thought of as the direction of the spiral of the edges which have $v$ as endpoints. 
In Figure \ref{Fig variety of ideal tri}, the enhancement $\varepsilon_i$ maps the closed geodesic boundary to the sign $(-1)^{i-1}$. 
The enhancement of $v$ appears as the sign of the sum of the shear parameters of the edges which have $v$ as endpoints (Proposition \ref{Prop b-len from shear}), however the shear parameters of the edges which have $v$ as both of the endpoints are added two times. 
\par
It is well known that the enhanced Teichm\"{u}ller space is parametrized by the shear parameters (\S 4.1 in \cite{Fock-Goncharov}). 
\begin{prop}\label{Prop shear coord}\cite{Fock-Goncharov}(Shear coordinates)
Let $\Gamma$ be a triangulation of $S$. 
Then the map 
$$\mathcal{T}^x(S)\rightarrow\mathbb{R}^{E^i(\Gamma)}$$
giving the shear parameters of the interior edges of $\Gamma$ is a global parametrization, where $E^i(\Gamma)$ is the set of the interior edges of $\Gamma$. 
\end{prop}
These coordinates are called the {\it shear coordinates}, and we give $\mathcal{T}^x(S)$ the differential structure of the coordinate space $\mathbb{R}^{E^i(\Gamma)}$. 

\subsection{Decorated Teichm\"{u}ller spaces}
\noindent
In this section, we consider the other data on the ends of the hyperbolic surfaces, which is called the {\it decoration}. 
We will define the deformation space respecting the marking and the decoration, and construct the $\lambda$-length coordinates of this space. 
\par
Suppose that hyperbolic surfaces in this section have no closed geodesic boundaries, and let $X$ be a hyperbolic surface. 
For a puncture (or spike) vertex $v$, lift the cusp (or the spike) $C_v$ corresponding to $v$ to the ideal point of the hyperbolic plane, and take a horocycle centered at the lift of $C_v$. 
The projection of the horocycle to $X$ is called a {\it horocycle} (or a {\it horocyclic arc}) around $C_v$. 
A {\it decoration curve} of $v$ is a horocycle centered at the cusp or a horocyclic arc centered at the spike corresponding to $v$. 
If a decoration curve around $v$ is a horocyclic arc, its endpoints are in the infinite geodesic boundaries which are asymptotic to the spike corresponding to $v$. 
\par
Only a part of the horocycle may be contained in $X$. 
For example, on the once punctured monogon in the left of Figure \ref{Fig deco curve}, the decoration curve $\gamma_1$ around $v_1$ is a part of horocycle. 
We define the length of the decoration curve as the length of the reconstructed horocycle. 
Glue the hyperbolic half planes along the infinite geodesic boundaries which intersects the decoration curve, and reconstruct the decoration curve by adding the horocyclic arcs on the glued half planes. 
Horocyclic arcs around spikes are regarded as the part of horocycles in half. 
For example, on the ideal triangle in the right of Figure \ref{Fig deco curve}, the length of the decoration curve $\gamma_1$ around $v_1$ is the double of the length of the horocyclic arc $\gamma_1$ (Figure \ref{Fig deco curve}). 
\begin{figure}[htbp]
\begin{center}
\includegraphics[width=14cm]{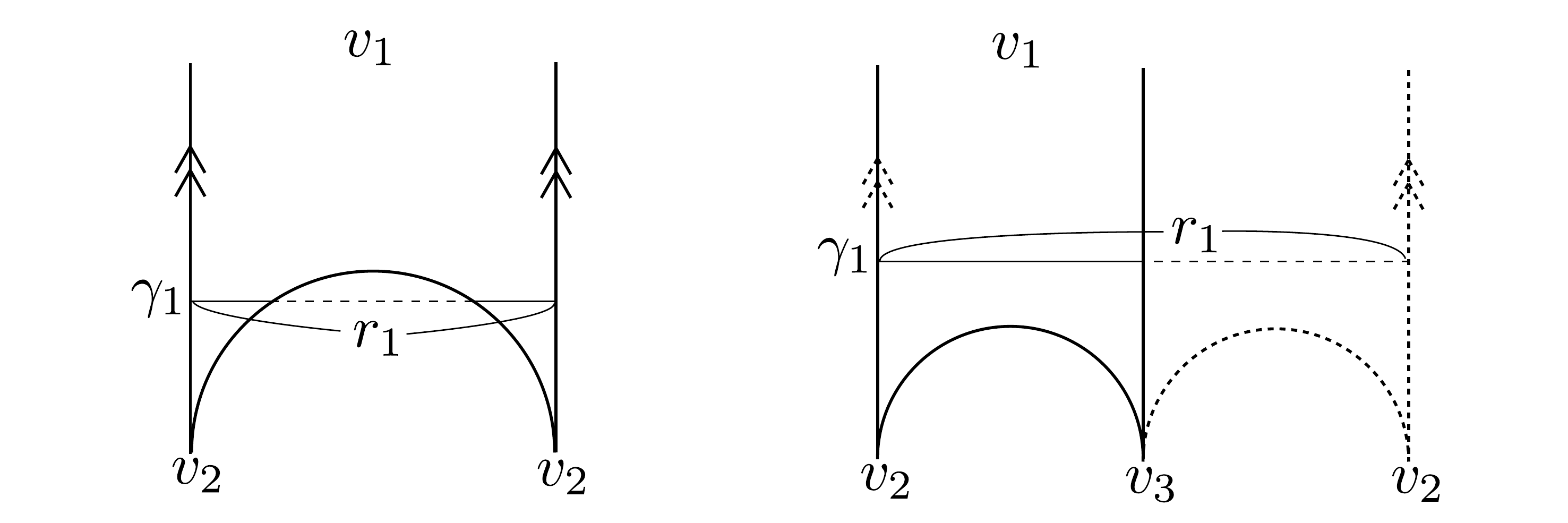}
\caption{A part of horocycle (left), a horocyclic arc (right) and their reconstructing. }
\label{Fig deco curve}
\end{center}
\end{figure}
\begin{dfn}\label{Def deco}(Decorations)
Let $f$ be a marking of $X$ and $D=\bigcup_{v\in V}\gamma_v$ be the union of the decoration curves $\gamma_v$ around the cusps or the spikes corresponding to the vertices $v$. 
Then $D$ is called a {\it decoration} of $X$ and a triplet $(X, f, D)$ is called a {\it decorated marked hyperbolic surface} of $S$. 
\end{dfn}
Decorated marked hyperbolic surfaces $(X_1, f_1, D_1), (X_2, f_2, D_2)$ are {\it decorated equivalent} and denoted by $(X_1, f_1, D_1)\sim_D(X_2, f_2, D_2)$, if $(X_1, f_1)\sim_T(X_2, f_2)$ and $\varphi(D_1)=D_2$ where $\varphi$ is the isometry preserving the marking. 
Let us define the deformation space of the marked hyperbolic structures with the decoration on $S$. 
\par
\begin{dfn}\label{Def deco Teich sp}(Decorated Teichm\"{u}ller spaces)
The set of decorated equivalence classes of decorated marked hyperbolic surfaces of $S$ is called the {\it decorated Teichm\"{u}ller space} of $S$, which is denoted by 
\begin{gather*}
\mathcal{T}^a(S)=\{\mathrm{Decorated\ marked\ hyperbolic\ surfaces\ of\ }S\\
\mathrm{without\ closed\ geodesic\ boundaries}\} /\sim_D. 
\end{gather*}
\end{dfn}
\begin{dfn}\label{Def l-length}($\lambda$-lengths)
Consider an edge $e$ of a triangulation $\Gamma$. 
Let $T, T'$ denote the triangles which have $e$ as a boundary, $v_1, v_2, v_3$ denote the vertices of $T$. 
Let $\gamma_1, \gamma_3$ be the decoration curves around $v_1, v_3$, respectively. 
Lift them to the upper half plane $\mathbb{H}^2$. 
The signed length along $\widetilde{e}$ from the intersection with $\widetilde{\gamma_3}$ to the base point of $\widetilde{e}$ with respect to $\widetilde{T}$ is called the {\it $\lambda$-length} of the half-edge $h$ of $e$ which has $v_3$ as the endpoint, where the direction respects to the orientation of the boundary of $\widetilde{T}$. 
Take a point $p$ on $\widetilde{e}$, and the sum of two signed lengths along $\widetilde{e}$ from the intersection with $\widetilde{\gamma_3}$ and $\widetilde{\gamma_1}$ to $p$ is called the {\it $\lambda$-length} of the edge $e$, where the directions respect to the orientations of the boundaries of $\widetilde{T}$ and $\widetilde{T'}$, respectively (Figure \ref{Fig l-length}). 
\end{dfn}
\begin{figure}[htbp]
\begin{center}
\includegraphics[width=14cm]{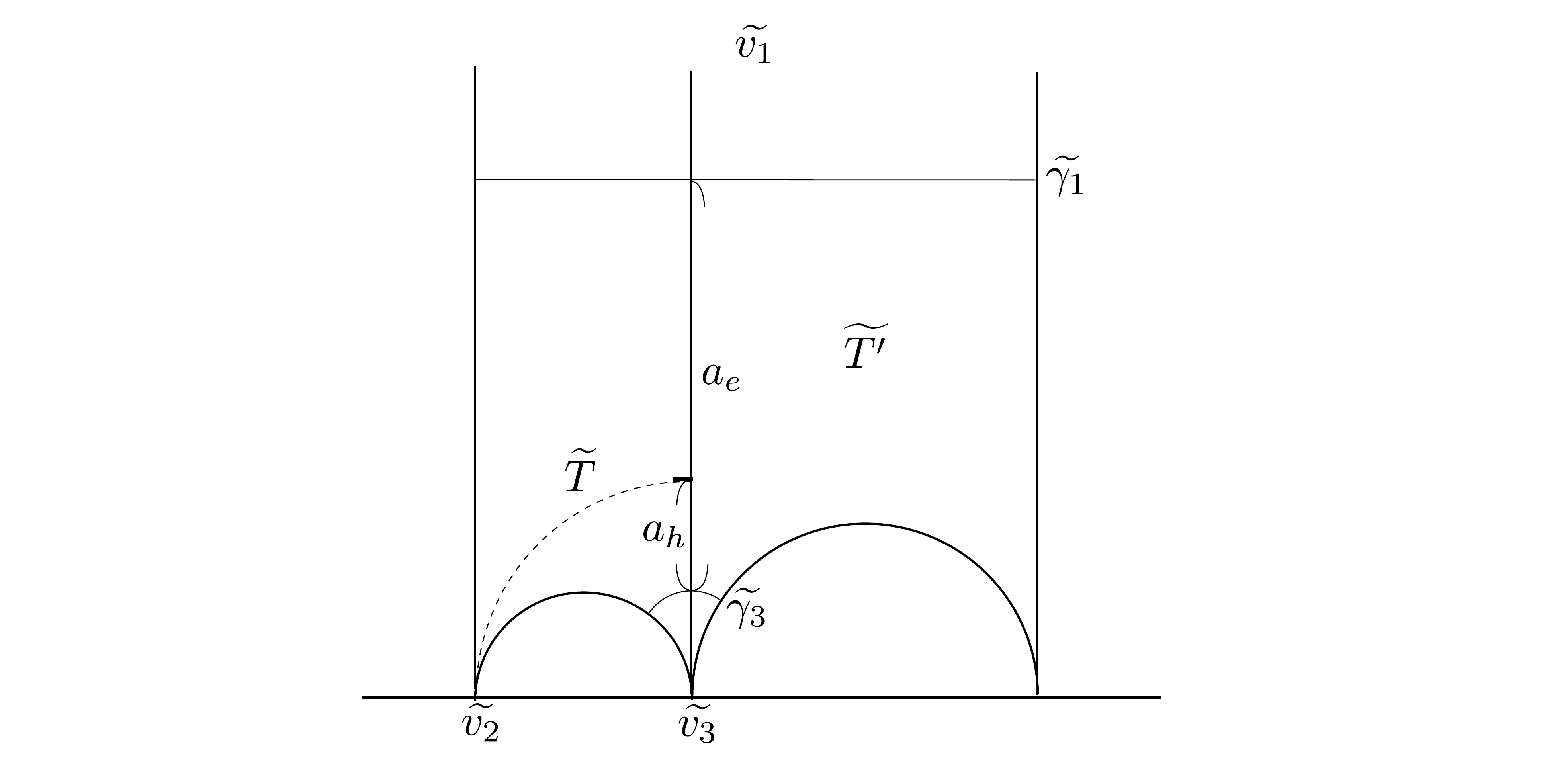}
\caption{The $\lambda$-length $a_e$ of an edge $e$, and the $\lambda$-length $a_h$ of a half-edge $h$. }
\label{Fig l-length}
\end{center}
\end{figure}
It is well known that the decorated Teichm\"{u}ller space is parametrized by the $\lambda$-length parameters (\S 4.2 in \cite{Fock-Goncharov}). 
\begin{prop}\label{Prop l-length coord}\cite{Fock-Goncharov}($\lambda$-length coordinates)
Let $\Gamma$ be a triangulation of $S$. 
Then the map 
$$\mathcal{T}^a(S)\rightarrow\mathbb{R}^{E(\Gamma)}$$
giving the $\lambda$-lengths of the edges of $\Gamma$ is a global parametrization, where $E(\Gamma)$ is the set of the edges of $\Gamma$. 
\end{prop}
These coordinates are called the {\it $\lambda$-length coordinates} and, we give $\mathcal{T}^a(S)$ the differential structure of the coordinate space $\mathbb{R}^{E(\Gamma)}$. 

\section{Decorated enhanced Teichm\"{u}ller spaces}
\noindent
In this chapter, we consider both of the additional data which are introduced before. 
We will define the deformation space respecting the marking, the enhancement and the decoration, and construct the shear-decoration coordinates of this space. 
This naturally generalized coordinates are obtained by adding the decoration parameters. 

\subsection{Decorated enhanced Teichm\"{u}ller spaces}
\noindent
The decorated enhanced Teichm\"{u}ller space is the quotient space of the enhanced marked hyperbolic surfaces which have the decorations.
We need to define the decoration curves around the closed geodesic boundaries. 
\par
Let $X$ be a hyperbolic surface. 
For a puncture vertex $v$, lift the closed geodesic boundary $\partial_v$ corresponding to $v$ to the hyperbolic plane, and take a curve on $\widetilde{X}$ which is equidistant from a lift $\widetilde{\partial_v}$ of $\partial_v$. 
The projection of the equidistant curve to $X$ is called an {\it equidistant curve} around $\partial_v$. 
Suppose that equidistant curves in this paper have positive distance from the corresponding closed geodesic boundaries. 
The same as horocycles, decoration curves may be parts of equidistant curves. 
The length of the part of equidistant curve is defined as the length of the reconstructed curve. 
\begin{prop}\label{Prop lim of eq dist curve}(The limit of the equidistant curves)
Let $v$ be a puncture vertex of $S$ and $X_l$ be a hyperbolic structure on reference surface $S$ whose closed geodesic boundary $\partial_l$ corresponding to $v$ has length $l$. 
Let $r$ be a positive number, and $\gamma_l$ be the length $r$ equidistant curve on $X_l$ around $\partial_l$. 
Then $\gamma_l$ converges to the horocycle of length $r$ around the cusp $C$ corresponding to $v$ when $l$ tends to $0$. 
\end{prop}
\begin{proof}
Lift $\partial_l$ and $\gamma_l$ to the upper half plane $\mathbb{H}^2$, and we can suppose that a lift $\widetilde{\partial_l}$ of $\partial_l$ and a lift $\widetilde{\gamma_l}$ of $\gamma_l$ have $-\frac{1}{l}$ and $\infty$ as the common endpoints (Figure \ref{Fig lim of eq dist curve}). 
Let $h_l$ denote the hyperbolic isometry of $\mathbb{H}^2$ corresponding to $v$ and 
$$h_l=
\begin{pmatrix}
e^{\frac{l}{2}}&\frac{1}{l}(e^{\frac{l}{2}}-e^{-\frac{l}{2}})\\
0&e^{-\frac{l}{2}}
\end{pmatrix}
\;\;\;\xrightarrow{l\to0}\;\;\;
\begin{pmatrix}
1&1\\
0&1
\end{pmatrix}
=:h_0. 
$$
The limit $h_0$ is the parabolic isometry of $\mathbb{H}^2$ corresponding to $v$ and the limit of $\partial_l$ is the cusp corresponding to $v$. 
\par
Let $\theta$ denote the angle of $\widetilde{\partial_l}$ and $\widetilde{\gamma_l}$. 
Take a parametrized curve $z(t)=(e^l-1)e^{i(\frac{\pi}{2}-\theta)}t+e^{i(\frac{\pi}{2}-\theta)}-\frac{1}{l}\;(0\le t\le1)$ and calculate $r$ by the line integral 
$$r=\int_0^1\frac{1}{y(t)}\sqrt{\left(\frac{dx}{dt}\right)^2+\left(\frac{dy}{dt}\right)^2}dt=\frac{l}{\cos\theta}. $$
Then, the curve $\widetilde{\gamma_l}$ is expressed as follows: 
$$\widetilde{\gamma_l}:y=\frac{l}{\sqrt{r^2-l^2}}x+\frac{1}{\sqrt{r^2-l^2}}\;\;\;\xrightarrow{l\to0}\;\;\;\widetilde{\gamma_0}:y=\frac{1}{r}. $$
Therefore, the limit curve of $\gamma_l$ is the projection $\gamma_0$ of $\widetilde{\gamma_0}$ and the length of $\gamma_0$ is $r$. 
\end{proof}
\begin{figure}[htbp]
\begin{center}
\includegraphics[width=14cm]{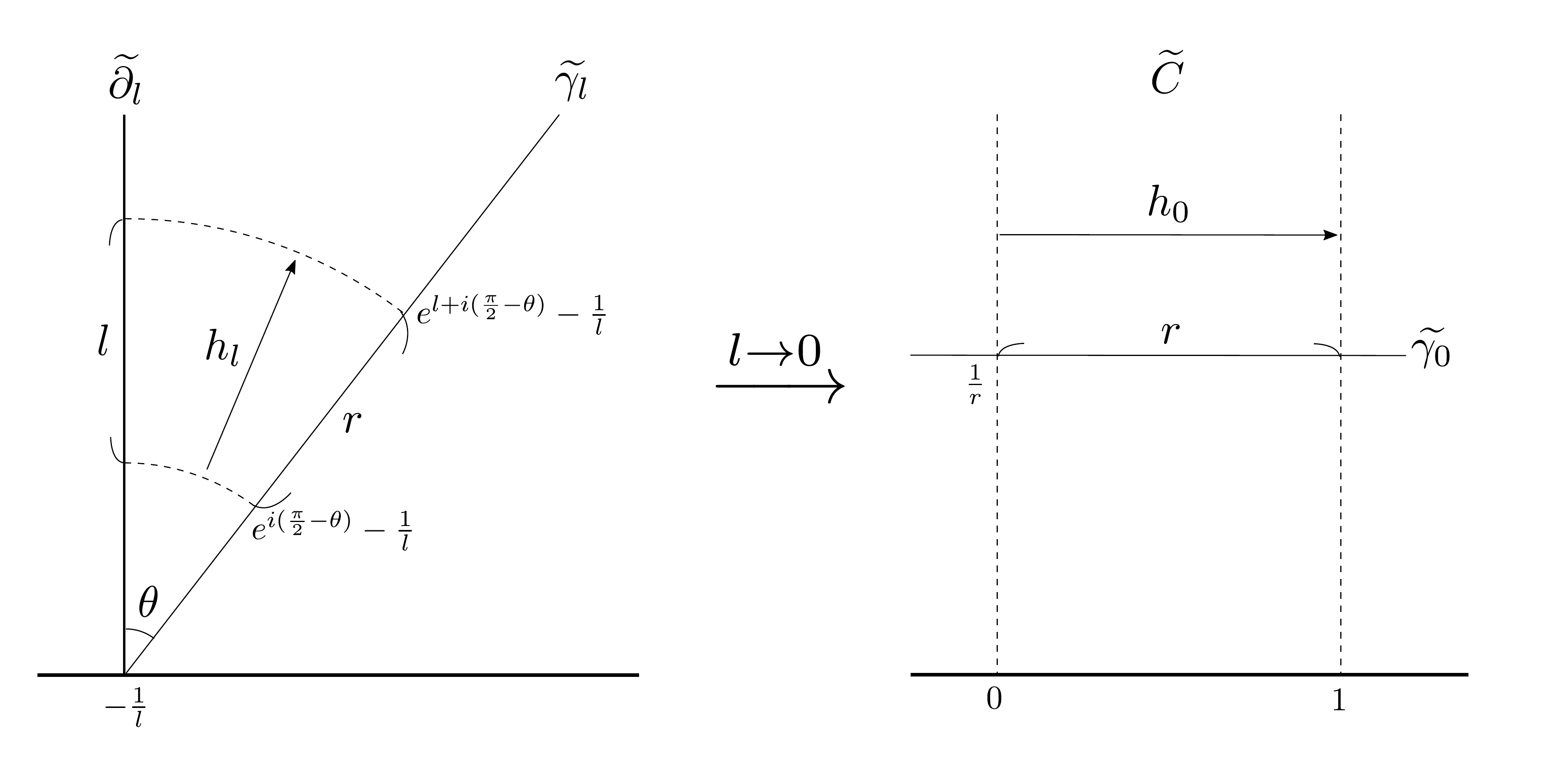}
\caption{The lift of an equidistant curve (left) and its limit (right). }
\label{Fig lim of eq dist curve}
\end{center}
\end{figure}
As seen from this proposition, equidistant curves have a role as the horocycles for the closed geodesic boundaries. 
For a vertex $v$, a {\it decoration curve} around $v$ is a horocycle around the cusp, a horocyclic arc around the spike or an equidistant curve around the closed geodesic boundary corresponding to $v$. 
\begin{dfn}\label{Def gene deco}(Generalized decorations)
Let $f$ be a marking of $X$, $\varepsilon$ be an enhancement of $X$, and $D=\bigcup_{v\in V}\gamma_v$ be the union of the decoration curves $\gamma_v$ around the cusps, the spikes or the closed geodesic boundaries corresponding to the vertices $v$. 
Then $D$ is called a {\it decoration} of $X$ and a quartet $(X, f, \varepsilon, D)$ is called a {\it decorated enhanced marked hyperbolic surface}. 
\end{dfn}
Decorated enhanced marked hyperbolic surfaces $(X_1, f_1, \varepsilon_1, D_1), (X_2, f_2, \varepsilon_2, D_2)$ are {\it decorated enhanced equivalent} and denoted by $(X_1, f_1, \varepsilon_1, D_1)\sim_{DE}(X_2, f_2, \varepsilon_2, D_2)$, if $(X_1, f_1, \varepsilon_1)\sim_{E}(X_2, f_2, \varepsilon_2)$ and $(X_1, f_1, D_1)\sim_{D}(X_2, f_2, D_2)$. 
Let us define the deformation space of the marked hyperbolic structures with the enhancement and the decoration on $S$. 
\begin{dfn}\label{Def deco enha Teich sp}(Decorated enhanced Teichm\"{u}ller spaces)
The set of decorated enhanced equivalence classes of decorated enhanced marked hyperbolic surfaces of $S$ is called the {\it decorated enhanced Teichm\"{u}ller space} of $S$, which is denoted by 
\begin{gather*}
\mathcal{T}^{ax}(S)=\{\mathrm{Decorated\ enhanced\ marked\ hyperbolic\ surfaces\ of\ }S\} /\sim_{DE}. 
\end{gather*}
\end{dfn}
Since the decoration is determined by the lengths of the decoration curves, there is an injection: 
$$\mathcal{T}^{ax}(S)\ni[X, f, \varepsilon, D]\mapsto([X, f, \varepsilon], (r_v)_{v\in V})\in\mathcal{T}^x(S)\times\mathbb{R}^V, $$
where $r_v$ is the length of the decoration curve corresponding to the vertex $v$. 
We give $\mathcal{T}^{ax}(S)$ the differential structure of the product structure of $\mathcal{T}^x(S)\times\mathbb{R}^V$. 

\subsection{Shear-decoration coordinates}
\noindent
The set of the decoration curves around a vertex $v$ is parametrized by their length, however this parameter has a lower bound which depends on the length of the boundary component corresponding to $v$. 
We define new parameters of the decoration and construct the generalized shear coordinates of the decorated enhanced Teichm\"{u}ller space. 
\par
Let $v$ be a vertex. 
If $v$ is a spike vertex, consider the doubled surface $S_D$, define the shear parameter of a boundary edge of $S$ as $0$, and define the shear parameter of the copy $e'$ of an interior edge $e$ of $S$ as the additive inverse of the shear parameter $x_e$ of $e$. 
Let $n$ denote the number of edges which have $v$ as an endpoint. 
Take an edge $e_1$ which has $v$ as an endpoint and take the edges $e_k$ inductively such that $e_{k+1}$ is the edge next to $e_k$ for counter-clockwise. 
Let $x_k$ denote the shear parameter of $e_k$, and $r_v$ denote the length of the decoration curve around $v$. 
\par
\begin{dfn}\label{Def deco para}(Decoration parameters)
the {\it decoration parameter} $d_v$ of $v$ is defined as follows. 
\begin{itemize}
\item If $v$ corresponds to a closed geodesic boundary, 
$$d_v=\frac{1}{2}\log\frac{\mathrm{min}_{1\le k\le n}\{r_k\}^2-l_v^2}{r_v^2-l_v^2}, $$
where $l_v$ is the length of the closed geodesic boundary corresponding to $v$ and 
$$r_k=\frac{l_v}{\exp(l_v)-1}\sqrt{(\sum_{i=1}^n\exp(\sum_{j=1}^{i-1}x_{k+j}))^2+(\exp(l_v)-1)^2}. $$
\item If $v$ corresponds to a cusp or a spike, 
$$d_v=\log\frac{\mathrm{min}_{1\le k\le n}\{r_k\}}{r_v}, $$
where 
$$r_k=\sum_{i=1}^n\exp(\sum_{j=1}^{i-1}x_{k+j}). $$
\end{itemize}
\end{dfn}
The cusp is thought of as the closed geodesic boundary whose length is $0$. 
We observe that 
$$\frac{l_v}{\mathrm{exp}(l_v)-1}\sqrt{(\sum_{i=1}^n\mathrm{exp}(\sum_{j=1}^{i-1}x_{k+j}))^2+(\mathrm{exp}(l_v)-1)^2}\xrightarrow{l_v\to 0}\sum_{i=1}^n\mathrm{exp}(\sum_{j=1}^{i-1}x_{k+j}) $$
and 
$$\frac{1}{2}\log\frac{\mathrm{min}_{1\le k\le n}\{r_k\}^2-l_v^2}{r_v^2-l_v^2}\xrightarrow{l_v\to 0}\log\frac{\mathrm{min}_{1\le k\le n}\{r_k\}}{r_v}. $$
By the decoration parameters, we obtain the generalized shear coordinates of the decorated enhanced Teichm\"{u}ller space. 
\par
\begin{thm}\label{Thm shear deco coord}(Shear-decoration coordinates)
Let $\Gamma$ be a triangulation of $S$. 
Then the decorated enhanced Teichm\"{u}ller space $\mathcal{T}^{ax}(S)$ of $S$ is parametrized by the shear parameters of the interior edges of $\Gamma$ and the decoration parameters around the vertices. 
Namely, the parameters give a homeomorphism: 
$$\psi_x : \mathcal{T}^{ax}(S)\rightarrow\mathbb{R}^{E^i(\Gamma)}\times\mathbb{R}^V. $$
\end{thm}
\begin{proof}
Let us regard $\mathcal{T}^{ax}(S)$ as the subset of $\mathbb{R}^{E^i(\Gamma)}\times\mathbb{R}^V$ by the shear parameters of the interior edges of $\Gamma$ and the lengths of the decoration curves around the vertices. 
We show that the map 
$$\psi_x : \mathcal{T}^{ax}(S)\ni((x_e)_{e\in E^i(\Gamma)}, (r_v)_{v\in V})\mapsto((x_e)_{e\in E^i(\Gamma)}, (d_v)_{v\in V})\in\mathbb{R}^{E^i(\Gamma)}\times\mathbb{R}^V $$
is a piecewise-diffeomorphism. 
\par
Let $v$ be a vertex. 
If $v$ is corresponding to a closed geodesic boundary, the range of the length of the decoration curve around $v$ is the interval $(l_v, \infty)$ and $d_v$ is the diffeomorphism from $(l_v, \infty)$ to $\mathbb{R}$, where $l_v$ is the length of the closed geodesic boundary corresponding to $v$. 
If $v$ is corresponding to a cusp or a spike, the range of the length of the decoration curve around $v$ is the interval $(0, \infty)$ and $d_v$ is the diffeomorphism from $(0, \infty)$ to $\mathbb{R}$. 
Therefore, $\psi_x$ is a bijection. 
\par
Consider the Jacobian matrix $D\psi_x$ of $\psi_x$. 
We can calculate the entries: 
$$\frac{\partial x_e}{\partial x_{e'}}=\left\{
\begin{array}{ll}
1&(e=e')\\
0&(e\neq e'), 
\end{array}
\right. $$
$$\frac{\partial x_e}{\partial r_v}=0, $$
\begin{equation}\label{Eq shear deco coord}
\frac{\partial d_v}{\partial r_{v'}}=\left\{
\begin{array}{ll}
-\frac{r_v}{r_v^2-l_v^2}&(v=v'\mathrm{\ corresponds\ to\ a\ closed\ geodesic\ boundary})\\
-\frac{1}{r_v}&(v=v'\mathrm{\ corresponds\ to\ a\ cusp\ or\ a\ spike})\\
0&(v\neq v'), 
\end{array}
\right. 
\end{equation}
and $D\psi_x$ is a lower triangular matrix. 
Since $\psi_x$ is bijective, it is a piecewise-diffeomorphism. 
\end{proof}
These coordinates are called the {\it shear-decoration coordinates}. 
\begin{remk}\label{Remk diff of psi}(Differentiability of $\psi_x$)
The decoration parameter introduced in Definition \ref{Def deco para} has a special origin for the theorem in \S 5.3. 
By changing $r_k$ into the value independent of the shear parameters, the map $\psi_x$ in Theorem \ref{Thm shear deco coord} becomes diffeomorphisms. 
\end{remk}

\subsection{Properties of decoration parameters}
\begin{prop}\label{Prop diff of l-len}(The differential of $\lambda$-lengths by decoration parameters)
Let $v$ be a vertex and $h$ be a half-edge which has $v$ as the endpoint. 
Then, 
$$\frac{da_h}{dd_v}=1. $$
\end{prop}
\begin{proof}
First, suppose that $v$ is corresponding to a closed geodesic boundary $\partial_v$. 
Let $l_v$ denote the length of $\partial_v$ and $\gamma_v$ be an equidistant curve around $\partial_v$. 
Lift them to the upper half plane $\mathbb{H}^2$ such that a lift $\widetilde{\partial_v}$ of $\partial_v$ and a lift $\widetilde{\gamma_v}$ of $\gamma_v$ have $0$ and $\infty$ as their common endpoints (Figure \ref{Fig diff of l-len 1}). 
Let $w$ denote the distance from $\partial_v$ to $\gamma_v$, $r_v$ denote the length of $\gamma_v$, and $\theta$ denote the angle of $\widetilde{\partial_v}$ and $\widetilde{\gamma_v}$ at $0$. 
Take a parametrized curve $z(t)=e^{it}\;(\frac{\pi}{2}-\theta\le t\le\frac{\pi}{2})$ and calculate $w$ by the line integral 
\begin{equation}\label{Eq diff of l-len 1}
w=\int_{\frac{\pi}{2}-\theta}^{\frac{\pi}{2}}\frac{1}{y(t)}\sqrt{\left(\frac{dx}{dt}\right)^2+\left(\frac{dy}{dt}\right)^2}dt=\log\frac{1+\sin\theta}{\cos\theta}. 
\end{equation}
Take a parametrized curve $z(t)=(e^{l_v}-1)e^{i(\frac{\pi}{2}-\theta)}t+e^{i(\frac{\pi}{2}-\theta)}\;(0\le t\le1)$ and calculate $r$ by the line integral 
\begin{equation}\label{Eq diff of l-len 2}
r_v=\int_0^1\frac{1}{y(t)}\sqrt{\left(\frac{dx}{dt}\right)^2+\left(\frac{dy}{dt}\right)^2}dt=\frac{l_v}{\cos\theta}. 
\end{equation}
By Equations (\ref{Eq diff of l-len 1}) and (\ref{Eq diff of l-len 2}), it is immediately follows that
$$r_v=l_v\cosh w. $$
Let $g_e$ be the geodesic corresponding to the edge $e$ which has $h$ as a half-edge. 
There is a lift $\widetilde{g_e}$ of $g_e$ which has $\infty$ as an endpoint (Figure \ref{Fig diff of l-len 1}). 
Similarly, the other endpoint $x \in \mathbb{R}$ of $\widetilde{g_e}$ is corresponding to the other endpoint $v' \in V$ of $e$. 
$\widetilde{g_e}$ intersects $\widetilde{\gamma_v}$ at $x+\frac{|x|}{\tan\theta}i$. 
The $\lambda$-length $a_h$ of the half-edge $h$ is the length between the intersection and the base point. 
We consider the infinitesimal deformation of $r_v$. 
If $r_v$ increases, $\theta$ increases, the imaginary part of the intersection of $\widetilde{g_e}$ and $\widetilde{\gamma_v}$ decreases, and $a_h$ decreases. 
The change of the amount of $a_h$ is 
$$\bigtriangleup a_h=\log\frac{\tan\theta}{\tan(\theta+\bigtriangleup\theta)}=\frac{1}{2}\log\frac{r_v^2-l_v^2}{(r_v+\bigtriangleup r_v)^2-l_v^2}. $$
Then the differential of $a_h$ by $r$ is 
$$\frac{da_h}{dr_v}=-\frac{r_v}{r_v^2-l_v^2}. $$
\begin{figure}[htbp]
\begin{center}
\includegraphics[width=15cm]{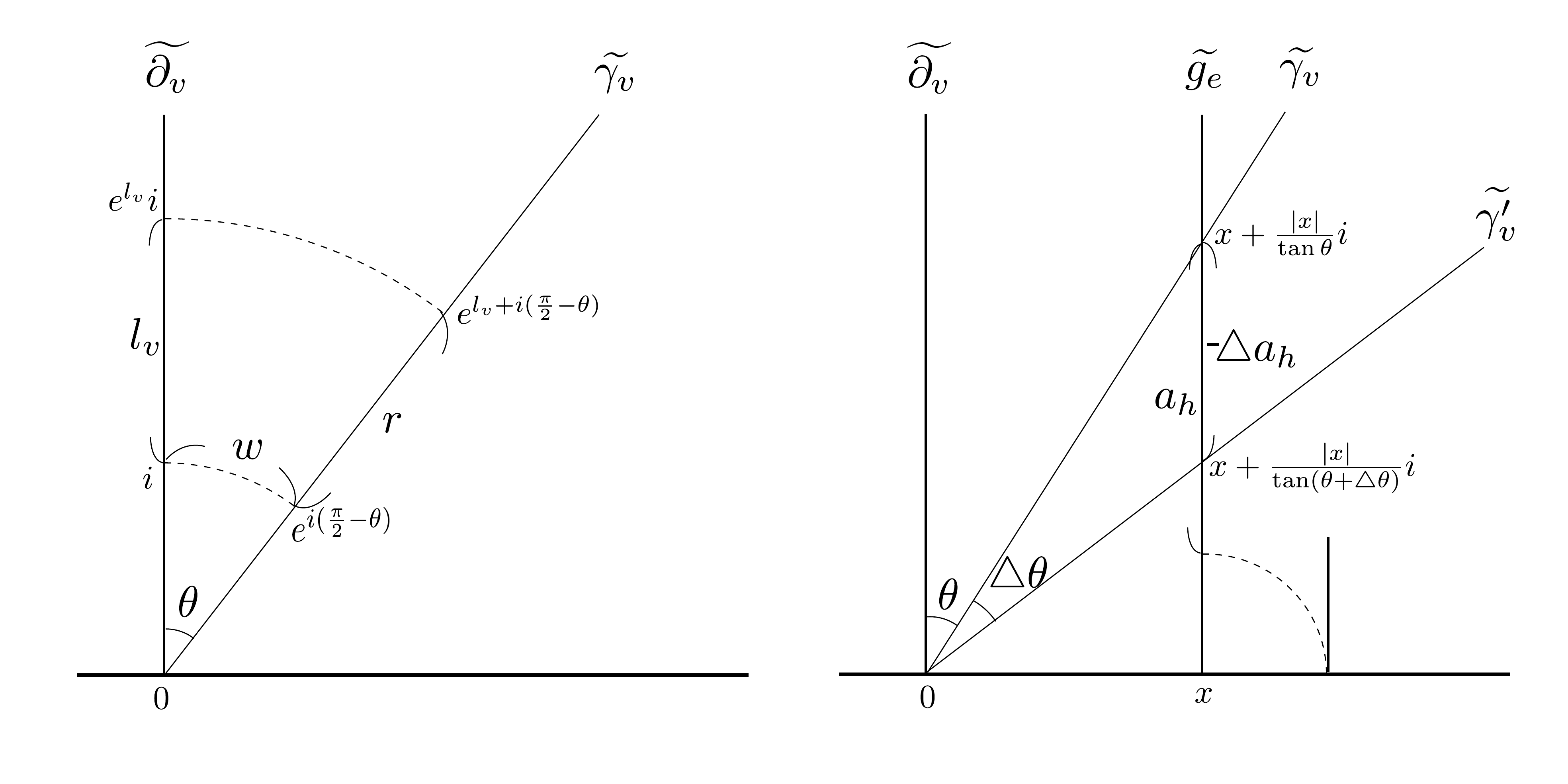}
\caption{The relation between $\bigtriangleup\theta$ and $\bigtriangleup a_h$ for the case of closed geodesic boundaries. }
\label{Fig diff of l-len 1}
\end{center}
\end{figure}
\par
Second, suppose that $v$ is corresponding to a cusp. 
Let $C_v$ denote the cusp corresponding to $v$ and $\gamma_v$ be a horocycle. 
Lift them to the upper half plane $\mathbb{H}^2$ such that a lift $\widetilde{C_v}$ of $C_v$ is at $\infty$ (Figure \ref{Fig diff of l-len 2}). 
Then there is a lift $\widetilde{\gamma_v}$ of $\gamma_v$ which is parallel to the real axis. 
Let $g_e$ be the geodesic corresponding to the edge $e$ which has $h$ as a half-edge. 
There is a lift $\widetilde{g_e}$ of $g_e$ which has $\infty$ as an endpoint and we can suppose that $\widetilde{g_e}$ and $f_v(\widetilde{g_e})$ have $0$ and $1$ as the other endpoint, respectively, where $f_v$ is the parabolic isometry of $\mathbb{H}^2$ corresponding to $C_v$. 
Let $y$ denote the imaginary part of $\widetilde{\gamma_v}$. 
Take a parameterized curve $z(t)=t+iy\;(0\le t\le1)$ and calculate the length $r_v$ of $\gamma_v$ by the line integral 
\begin{equation}\label{Eq diff of l-len 3}
r_v=\int_0^1\frac{1}{y(t)}\sqrt{\left(\frac{dx}{dt}\right)^2+\left(\frac{dy}{dt}\right)^2}dt=\frac{1}{y}. 
\end{equation}
We consider the $\lambda$-length $a_h$ of the half-edge $h$ for the infinitesimal deformation of $r_v$. 
If $r_v$ increases, the imaginary part $y$ of the intersection of $\widetilde{g_e}$ and $\widetilde{\gamma_v}$ decreases, and $a_h$ decreases. 
The change of the amount of $a_h$ is 
$$\bigtriangleup a_h=\log\frac{y+\bigtriangleup y}{y}=\log\frac{r_v}{r_v+\bigtriangleup r_v}. $$
Then the differential of $a_h$ by $r_v$ is 
$$\frac{da_h}{dr_v}=-\frac{1}{r_v}. $$
\par
Finally, suppose that $v$ is corresponding to a spike. 
This case is similar to the case of the cusps. 
Let $C_v$ denote the spike corresponding to $v$ and $\gamma_v$ be a horocyclic arc. 
Lift them to the upper half plane $\mathbb{H}^2$ such that a lift $\widetilde{C_v}$ of $C_v$ is at $\infty$ (Figure \ref{Fig diff of l-len 2}). 
Then there is a lift $\widetilde{\gamma_v}$ of $\gamma_v$ which is parallel to the real axis. 
Let $g_0$ and $g_1$ denote the infinite geodesic boundaries of $X$ ending at $C_v$, there are lifts $\widetilde{g_0}$ of $g_0$ and $\widetilde{g_1}$ of $g_1$ which have $\infty$ as an endpoint, and we can suppose that $\widetilde{g_0}$ and $\widetilde{g_1}$ have $0$ and $1$ as the other endpoint, respectively. 
Let $g_e$ be the geodesic corresponding to the edge $e$ which has $h$ as a half-edge, then there is a lift $\widetilde{g_e}$ of $g_e$ which has $\infty$ as an endpoint. 
Let $y$ be the imaginary part of $\widetilde{\gamma_v}$. 
Calculate the length $r$ of $\gamma_v$ by the line integral 
$$r_v=2\int_0^1\frac{1}{y(t)}\sqrt{\left(\frac{dx}{dt}\right)^2+\left(\frac{dy}{dt}\right)^2}dt=\frac{2}{y}. $$
By the same calculation as the case of the cusps, the differential of the $\lambda$-length $a_h$ of the half-edge $h$ by $r_v$ is 
$$\frac{da_h}{dr_v}=-\frac{1}{r_v}. $$
\par
\begin{figure}[htbp]
\begin{center}
\includegraphics[width=15cm]{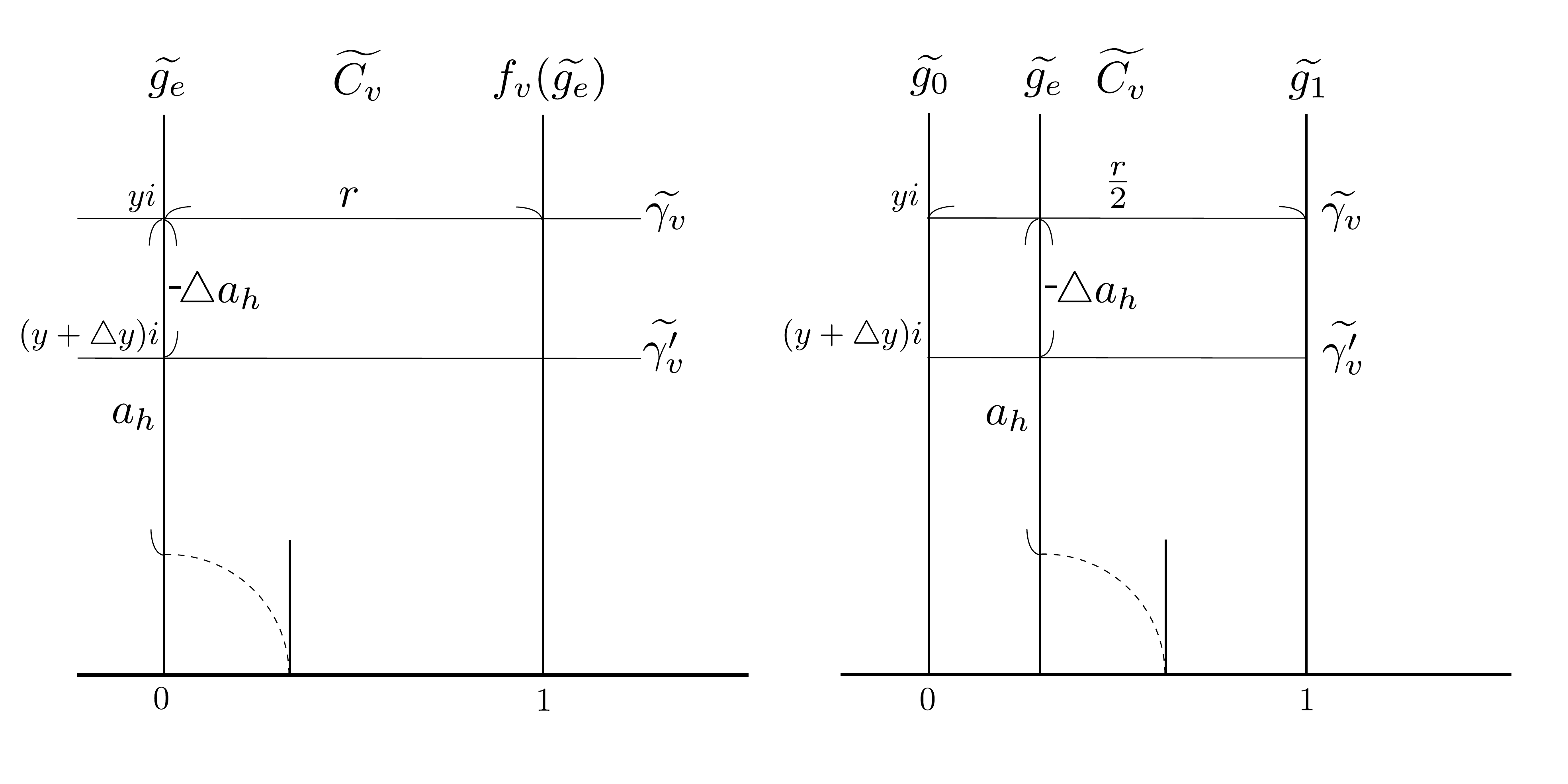}
\caption{Relation between $\bigtriangleup y$ and $\bigtriangleup a_h$ for the case of the cusps (left) and the spikes (right). }
\label{Fig diff of l-len 2}
\end{center}
\end{figure}
In comparison with Equation (\ref{Eq shear deco coord}), the assigned equation follows. 
\end{proof}
The $\lambda$-length $a_h$ of a half-edge $h$ which does not have a vertex $v$ as the endpoint is independent of the decoration parameter $d_v$ of $v$. 
Then, we have that 
$$\frac{da_h}{dd_v}=0$$
for such a half-edge $h$. 
\par
Finally, we research the origins of the decoration parameters. 
Decoration curves are the boundaries of the equidistant neighborhoods around closed geodesic boundaries, or the cusp neighborhoods around cusps or spikes. 
$r_k$ in Definition \ref{Def deco para} is the length of the boundary curve of the equidistant (or cusp) neighborhoods whose boundary curve passes through the base point of the triangle which has $e_k$ as a boundary. 
By Definition \ref{Def deco para}, the decoration parameter of $v$ is $0$ if and only if the length of the curve $r_v$ equals the shortest boundary curve of such neighborhoods. 
Consequently, following proposition follows. 
\begin{prop}\label{Prop origin of deco para}(Origins of decoration parameters)
Let $v$ be a vertex. 
Then, the following conditions are equivalent. 
\begin{itemize}
\item
$d_v(r_v)=0. $
\item
The decoration curve around $v$ is the boundary curve of the equidistant (or cusp) neighborhood which is intersection of the equidistant (or cusp) neighborhoods whose boundary curve passes through the base point of the triangle which has $v$ as a vertex. 
\end{itemize}
\end{prop}

\section{$\lambda$-lengths and boundary lengths}
\noindent
For an element of the decorated enhanced Teichm\"{u}ller space, the $\lambda$-length is defined as in the decorated Teichm\"{u}ller space and the signed boundary length is defined as in the enhanced Teichm\"{u}ller space. 
In this chapter, we study the relation between these parameters and the shear-decoration coordinates, and construct the generalized $\lambda$-length coordinates of the decorated enhanced Teichm\"{u}ller space. 

\subsection{Calculations from shear-decoration coordinates}
\noindent
In this section, we calculate the $\lambda$-lengths and the signed boundary lengths from the shear parameters and the decoration parameters. 
These calculations make the map from $\mathbb{R}^{E^i(\Gamma)}\times\mathbb{R}^V$ to $\mathbb{R}^{E(\Gamma)}\times\mathbb{R}^{V^i}$. 
\par
Let $v$ be a vertex. 
If $v$ is a spike vertex, consider the doubled surface $S_D$, define the shear parameter of a boundary edge of $S$ as $0$, and define the shear parameter of the copy $e'$ of an interior edge $e$ of $S$ as the additive inverse of the shear parameter $x_e$ of $e$. 
Let $n$ denote the number of edges which have $v$ as an endpoint. 
Take an edge $e_1$ which has $v$ as an endpoint and take the edges $e_k$ inductively such that $e_{k+1}$ is the edge next to $e_k$ for counter-clockwise. 
Let $x_k$ denote the shear parameter of $e_k$, let $r_v$ denote the length of the decoration curve around $v$, and let $g_k$ denote the geodesic corresponding to $e_k$. 
Lift these geodesics to the upper half plane $\mathbb{H}^2$ such that a lift $\widetilde{g_k}$ of $g_k$ has $\infty$ as an endpoint. 
\begin{lem}\label{Lem real part of endpt}(Real parts of endpoints)
For $1\le k\le n+1$, the other endpoint $u_k$ of $\widetilde{g_k}$ satisfies 
$$u_k=u_1+(u_{n+1}-u_1)\frac{\sum_{i=1}^{k-1}\exp(\sum_{j=1}^{i-1}x_{j+1})}{\sum_{i=1}^{n}\exp(\sum_{j=1}^{i-1}x_{j+1})}. $$
\end{lem}
\begin{proof}
By definition, $x_k$ is the signed length between the base point of $\widetilde{g_k}$ with respect to the triangle which has $u_{k-1}$ as a vertex and the base point of $\widetilde{g_k}$ with respect to the triangle which has $u_{k+1}$ as a vertex. 
Then, for an integer $k$, 
\begin{equation}\label{Eq real part of endpt}
x_k=\log\frac{u_{k+1}-u_k}{u_k-u_{k-1}}. 
\end{equation}
This is the recurrence formula of $\{u_k\}_{k\in\mathbb{Z}}$, and 
\begin{eqnarray*}
u_{n+1}-u_1\!\!\!&=\!\!\!&\sum_{i=1}^n(u_{i+1}-u_i)\\
&=\!\!\!&\sum_{i=1}^n(u_2-u_1)\exp(\sum_{j=1}^{i-1}x_{j+1}), \\
u_k\!\!\!&=\!\!\!&u_1+\sum_{i=1}^{k-1}(u_{i+1}-u_i)\\
&=\!\!\!&u_1+\sum_{i=1}^{k-1}(u_2-u_1)\exp(\sum_{j=1}^{i-1}x_{j+1})\\
&=\!\!\!&u_1+(u_{n+1}-u_1)\frac{\sum_{i=1}^{k-1}\exp(\sum_{j=1}^{i-1}x_{j+1})}{\sum_{i=1}^{n}\exp(\sum_{j=1}^{i-1}x_{j+1})}. 
\end{eqnarray*}
\end{proof}
To claim the relation between the shear parameters and the signed boundary lengths, suppose that $v$ is corresponding to a closed geodesic boundary and $n$-valent. 
Let $\partial_v$ denote the closed geodesic boundary corresponding to $v$ and $l_v$ denote the signed length of $\partial_v$, where the sign of $l_v$ is the enhancement $\varepsilon(\partial_v)$ of $\partial_v$. 
Then, $f_v(z)=e^{l_v}z$ is the hyperbolic isometry of $\mathbb{H}^2$ corresponding to $v$. 
For $i\in\{0, 1\}$, $f_v(\widetilde{g_i})=\widetilde{g_{n+i}}$ and the geodesics $g_i$ and $g_{n+i}$ are identified. 
\begin{prop}\label{Prop b-len from shear}(Signed boundary lengths from shear parameters)
The shear parameters and the signed boundary lengths satisfy 
$$l_v=\sum_{i=1}^nx_i. $$
\end{prop}
\begin{proof}
The endpoint $u_k$ of the lift $\widetilde{g_k}$ of $g_k$ is positive if and only if $\varepsilon(\partial_v)$ is positive. 
By Equation (\ref{Eq real part of endpt}), 
\begin{eqnarray*}
e^{l_v}\!\!\!&=\!\!\!&\frac{e^{l_v}u_1-e^{l_v}u_0}{u_1-u_0}=\frac{f_v(u_1)-f_v(u_0)}{u_1-u_0}=\frac{u_{n+1}-u_n}{u_1-u_0}\\
&=\!\!\!&\prod_{i=1}^n\frac{u_{i+1}-u_i}{u_i-u_{i-1}}=\prod_{i=1}^ne^{x_i}=\exp(\sum_{i=1}^nx_i)
\end{eqnarray*}
and the assigned equation follows. 
\end{proof}
For a vertex $v$ corresponding to a cusp, the parabolic isometry $f_v$ of $\mathbb{H}^2$ corresponding to $v$ is the parallel translation. 
Then 
\begin{eqnarray*}
\exp(\sum_{i=1}^nx_i)&=\!\!\!&\prod_{i=1}^ne^{x_i}=\prod_{i=1}^n\frac{u_{i+1}-u_i}{u_i-u_{i-1}}\\
&=\!\!\!&\frac{u_{n+1}-u_n}{u_1-u_0}=\frac{f_v(u_1)-f_v(u_0)}{u_1-u_0}=1, 
\end{eqnarray*}
and Proposition \ref{Prop b-len from shear} holds. 
\begin{prop}\label{Prop l-len from shear and deco}($\lambda$-lengths from shear parameters and decoration parameters)
For a half-edge $h$ which has $v$ as the endpoint, the $\lambda$-length of $h$ is as follows: 
$$a_h=d_v+\log\frac{\sum_{i=1}^n\exp(\sum_{j=1}^{i-1}x_{j+1})}{\mathrm{min}_{1\le k\le n}\{\sum_{i=1}^n\exp(\sum_{j=1}^{i-1}x_{k+j})\}}, $$
where $d_v$ is the decoration parameter of $v$. 
Then, the $\lambda$-length of an edge $e$ is as follows: 
$$a_e=x_e+\!\!\sum_{h\in E^h(e)}\!\!a_h, $$
where $E^h(e)$ is the set of the half-edges of $e$. 
\end{prop}
\begin{proof}
First, consider the case that $v$ is corresponding to a closed geodesic boundary $\partial_v$. 
Suppose that $h$ is a half-edge of $e_1$ in the beginning of this section. 
Lift the geodesic $g_k$ corresponding to $e_k$ to the upper half plane $\mathbb{H}^2$ as given in Lemma \ref{Lem real part of endpt}. 
We can suppose that a lift $\widetilde{\partial_v}$ of $\partial_v$ and a lift $\widetilde{\gamma_v}$ of the decoration curve $\gamma_v$ have the common endpoints $0$ and $\infty$, $u_1=\varepsilon(\partial_v)$ and $u_{n+1}=\varepsilon(\partial_v)e^{l_v}$, where $\varepsilon$ is the enhancement and $l_v$ is the signed length of $\partial_v$. 
Let $\theta$ denote the angle of $\widetilde{\partial_v}$ and $\widetilde{\gamma_v}$. 
By Definition \ref{Def deco para}, Equation (\ref{Eq diff of l-len 2}), Lemma \ref{Lem real part of endpt} and Proposition \ref{Prop b-len from shear}, 
\begin{eqnarray*}
a_h\!\!\!&=\!\!\!&\log\frac{|u_1|\frac{1}{\tan\theta}}{u_2-u_1}\\
&=\!\!\!&\log\frac{l_v\sum_{i=1}^n\exp(\sum_{j=1}^{i-1}x_{j+1})}{(\exp(l_v)-1)\sqrt{r_v^2-l_v^2}}\\
&=\!\!\!&d_v+\log\frac{\sum_{i=1}^n\exp(\sum_{j=1}^{i-1}x_{j+1})}{\mathrm{min}_{1\le k\le n}\{\sum_{i=1}^n\exp(\sum_{j=1}^{i-1}x_{k+j})\}}. 
\end{eqnarray*}
\par
The argument for the cases of  the cusps and the spikes are similar. 
Lift the geodesic $g_k$ corresponding to $e_k$ to the upper half plane $\mathbb{H}^2$ as given in Lemma \ref{Lem real part of endpt}, and we can suppose that $u_1=0$ and $u_{n+1}=1$. 
Let $y$ denote the imaginary part of the lift $\widetilde{\gamma}$ of the decoration curve $\gamma$. 
By Definition \ref{Def deco para}, Equation (\ref{Eq diff of l-len 3}) and Lemma \ref{Lem real part of endpt}, 
\begin{eqnarray*}
a_h\!\!\!&=\!\!\!&\log\frac{y}{u_2-u_1}\\
&=\!\!\!&\log\frac{\sum_{i=1}^n\exp(\sum_{j=1}^{i-1}x_{j+1})}{r_v}\\
&=\!\!\!&d_v+\log\frac{\sum_{i=1}^n\exp(\sum_{j=1}^{i-1}x_{j+1})}{\mathrm{min}_{1\le k\le n}\{\sum_{i=1}^n\exp(\sum_{j=1}^{i-1}x_{k+j})\}}. 
\end{eqnarray*}
\par
Finally, we research the calculation of $a_e$. 
By definition, the shear parameter $x_e$ of $e$ is the signed length between two base points of $e$, and two $\lambda$-lengths of half-edges are the signed lengths between the respective base point and the intersection with the respective decoration curve. 
Then, the sum of these three lengths equals the $\lambda$-length $a_e$. 
\end{proof}
By Proposition \ref{Prop b-len from shear} and \ref{Prop l-len from shear and deco}, we can construct the map 
$$\psi : \mathbb{R}^{E^i(\Gamma)}\times\mathbb{R}^V\ni((x_e)_{e\in E^i(\Gamma)}, (d_v)_{v\in V})\mapsto((a_e)_{e\in E(\Gamma)}, (l_v)_{v\in V^i})\in\mathbb{R}^{E(\Gamma)}\times\mathbb{R}^{V^i}, $$
where $(a_e)_{e\in E(\Gamma)}$ and $(l_v)_{v\in V^i}$ are the $\lambda$-lengths and signed boundary lengths of the decorated enhanced marked hyperbolic surface whose shear-decoration coordinates are $((x_e)_{e\in E^i(\Gamma)}, (d_v)_{v\in V})$. 

\subsection{$\lambda$-boundary-length coordinates}
\noindent
In this section, we calculate the shear parameters and the decoration parameters from the $\lambda$-lengths and the signed boundary lengths. 
These calculations make the inverse map from $\mathbb{R}^{E(\Gamma)}\times\mathbb{R}^{V^i}$ to $\mathbb{R}^{E^i(\Gamma)}\times\mathbb{R}^V$ and give $\mathcal{T}^{ax}(S)$ the generalized $\lambda$-length coordinates. 
\par
\begin{thm}\label{Thm l-b-len coord}($\lambda$-boundary-length coordinates)
For a triangulation $\Gamma_0$ whose puncture vertices are 1-valent, the decorated enhanced Teichm\"{u}ller space $\mathcal{T}^{ax}(S)$ of $S$ is parametrized by the $\lambda$-lengths of the edges of $\Gamma_0$ and the signed boundary lengths of the puncture vertices. 
Namely, the parameters give a homeomorphism: 
$$\psi_a : \mathcal{T}^{ax}(S)\rightarrow\mathbb{R}^{E(\Gamma_0)}\times\mathbb{R}^{V^i}. $$
\end{thm}
These coordinates are called the {\it $\lambda$-boundary-length coordinates}. 
Theorem \ref{Thm l-b-len coord} is the main theorem of this paper. 
In the remains of this section, we prove this theorem. 
\par
Let $e$ be an interior edge. 
We research the relation between the shear parameter of $e$ and the $\lambda$-lengths. 
Let $T, T'$ be the triangles sharing $e$ as a boundary geodesic, $v_1, v_2, v_3$ be the vertices of $T$, $v_3, v_4, v_1$ be the vertices of $T'$, and $e_{i, j}$ be the boundaries of $T$ or $T'$ which have $v_i$ and $v_j$ as the endpoints. 
Conveniently, regard these subscripts as the numbers modulo $4$. 
Lift them to the upper half plane $\mathbb{H}^2$. 
Let $p_i$ denote the intersections of $\widetilde{e_{i, i+1}}$ and the lifts of the decoration curves around $v_i$. 
The gaps $t_i$ denote the signed lengths along $\widetilde{e_{i-1, i}}$ from the intersections with the horocyclic arcs which intersect $\widetilde{e_{i, i+1}}$ at $p_i$ to the intersections with the lifts of the decoration curves around $v_i$, where the directions respect to the orientations of the boundaries of $\widetilde{T}$ or $\widetilde{T'}$ (Figure \ref{Fig shear from l-len and gap}). 
\begin{lem}\label{Lem gap from shear}(Gaps from shear parameters)
For $1\le k\le 4$, let $n_k$ be the valence of $v_k$, let $x_{k, 1}$ be the shear parameter of $e_{k, k+1}$, and let $x_{k, j+1}$ be the shear parameters next to $x_{k, j}$ for counter-clockwise inductively. 
If $v_k$ is corresponding to a closed geodesic boundary and $k$ is even, 
$$t_k=x_{k, 2}+\log\frac{\sum_{i=1}^{n_k}\exp(\sum_{j=1}^{i-1}x_{k, j+2})}{\sum_{i=1}^{n_k}\exp(\sum_{j=1}^{i-1}x_{k, j+1})}. $$
If $v_k$ is corresponding to a closed geodesic boundary and $k$ is odd, 
$$t_k=x_{k, 2}+x_{k, 3}+\log\frac{\sum_{i=1}^{n_k}\exp(\sum_{j=1}^{i-1}x_{k, j+3})}{\sum_{i=1}^{n_k}\exp(\sum_{j=1}^{i-1}x_{k, j+1})}. $$
If $v_k$ is corresponding to a cusp or a spike, 
$$t_k=0. $$
\end{lem}
\begin{proof}
By definition, if $v_k$ is corresponding to a cusp or a spike, the decoration curve around $v_k$ is a horocyclic arc and the gap $t_k$ is $0$. 
\par
We consider the case of the closed geodesic boundaries. 
Lift the geodesics $g_j$ corresponding to the edge which gives the shear parameters $x_{k, j}$ to the upper half plane $\mathbb{H}^2$ as given in Lemma \ref{Lem real part of endpt}. 
We can suppose that a lift $\widetilde{\partial_k}$ of the boundary $\partial_k$ and a lift $\widetilde{\gamma_k}$ of the decoration curve $\gamma_k$ around $v_k$ have the common endpoints $0$ and $\infty$, $u_1=\varepsilon(\partial_v)$ and $u_{n+1}=\varepsilon(\partial_v)e^{l_v}$, where $\varepsilon$ is the enhancement and $l_v$ is the signed length of $\partial_v$. 
Let $\theta$ denote the angle of $\widetilde{\partial_v}$ and $\widetilde{\gamma_v}$. 
If $k$ is even, then $e_{k-1, k}$ corresponds to $g_2$ and 
\begin{eqnarray*}
t_k\!\!\!&=\!\!\!&\log\frac{|u_2|\frac{1}{\tan\theta}}{\frac{1}{\tan\theta}}\\
&=\!\!\!&x_{k, 2}+\log\frac{\sum_{i=1}^{n_k}\exp(\sum_{j=1}^{i-1}x_{k, j+2})}{\sum_{i=1}^{n_k}\exp(\sum_{j=1}^{i-1}x_{k, j+1})}. 
\end{eqnarray*}
If $k$ is odd, then $e_{k-1, k}$ corresponds to $g_3$ and 
\begin{eqnarray*}
t_k\!\!\!&=\!\!\!&\log\frac{|u_3|\frac{1}{\tan\theta}}{\frac{1}{\tan\theta}}\\
&=\!\!\!&x_{k, 2}+x_{k, 3}+\log\frac{\sum_{i=1}^{n_k}\exp(\sum_{j=1}^{i-1}x_{k, j+3})}{\sum_{i=1}^{n_k}\exp(\sum_{j=1}^{i-1}x_{k, j+1})}. 
\end{eqnarray*}
\end{proof}
\begin{lem}\label{Lem shear from l-len and gap}(Shear parameters from $\lambda$-lengths and gaps)
The shear parameter of $e$ is 
$$x_e=\frac{1}{2}(a_{1, 2}-a_{2, 3}+a_{3, 4}-a_{4, 1}+t_1-t_2+t_3-t_4), $$
where $a_{i, j}$ are the $\lambda$-lengths of $e_{i, j}$. 
\end{lem}
\begin{proof}
For $1\le k\le 4$, the $\lambda$-length $a_{k, k+1}$ of $e_{k, k+1}$ is the sum of two signed lengths from the base points of $\widetilde{T}$ or $\widetilde{T'}$ on $\widetilde{e_{k, k+1}}$ to the intersections with the lifts of the decoration curves around $v_k$ and $v_{k+1}$. 
Let $a_k$ denote the first term of this sum $a_{k, k+1}$. 
Then, 
\begin{eqnarray*}
a_{1, 2}\!\!\!&=\!\!\!&a_1+(a_2+t_2), \\
a_{2, 3}\!\!\!&=\!\!\!&a_2+(a_3+t_3-x_e), \\
a_{3, 4}\!\!\!&=\!\!\!&a_3+(a_4+t_4), \\
a_{4, 1}\!\!\!&=\!\!\!&a_4+(a_1+t_1-x_e), 
\end{eqnarray*}
and the assigned equation follows. 
\end{proof}
\begin{figure}[htbp]
\begin{center}
\includegraphics[width=15cm]{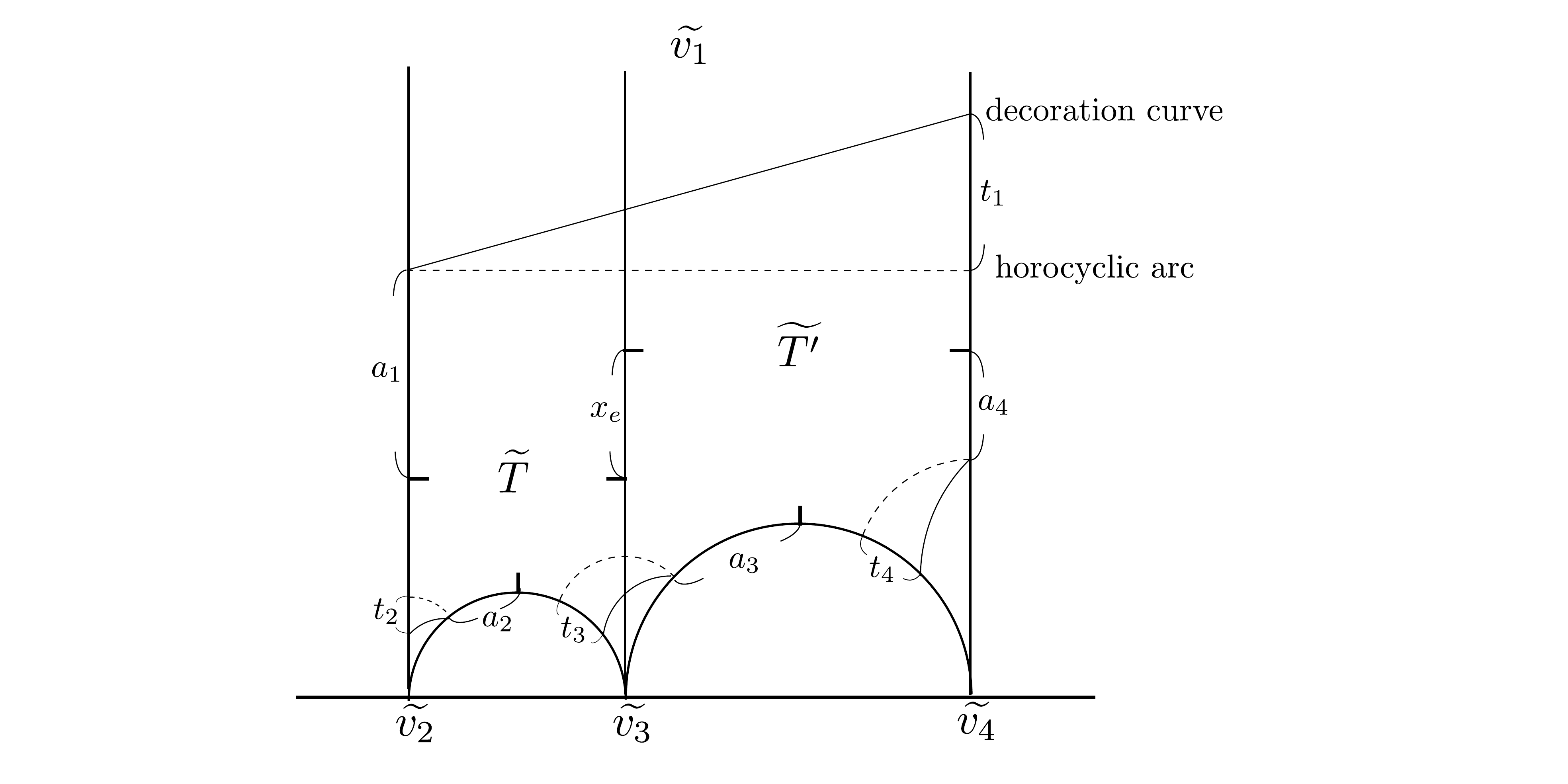}
\caption{Gaps between the horocyclic arcs and the decoration curves. }
\label{Fig shear from l-len and gap}
\end{center}
\end{figure}
Lemma \ref{Lem shear from l-len and gap} gives the relation between the shear parameters and the $\lambda$-lengths, however we need to solve the equations because the gaps $t_k$ are expressed by the shear parameters. 
By Proposition \ref{Prop b-len from shear} and Lemma \ref{Lem gap from shear}, $t_k$ are $0$ if $v_k$ are corresponding to cusps or spikes, and $t_k$ are the signed boundary lengths of $v_k$ if the valences of $v_k$ are $1$. 
The equations are solvable for the shear parameters when the triangulation is special. 
\begin{prop}\label{Prop shear from l-b-len}(Shear parameters from $\lambda$-lengths and signed boundary lengths)
Suppose that any vertex of $V^i$ is 1-valent, then it is classified into three cases: 
\begin{itemize}
\item $v_1$ is a puncture vertex, $v_2, v_3, v_4$ are spike vertices, and
$$x_e=\frac{1}{2}(a_{1, 2}-a_{2, 3}+a_{3, 4}-a_{4, 1}+2l_1), $$
where $l_1$ is the signed boundary length of $v_1$. 
\item $v_2$ is a puncture vertex, $v_1, v_3, v_4$ are spike vertices, and
$$x_e=\frac{1}{2}(a_{1, 2}-a_{2, 3}+a_{3, 4}-a_{4, 1}-l_2), $$
where $l_2$ is the signed boundary length of $v_2$. 
\item $v_1, v_2, v_3, v_4$ are spike vertices, and
$$x_e=\frac{1}{2}(a_{1, 2}-a_{2, 3}+a_{3, 4}-a_{4, 1}). $$
\end{itemize}
\end{prop}
\begin{proof}
First, suppose that $v_1$ is a  puncture vertex. 
By the assumption, the valence of $v_1$ is $1$, then $e=e_{1, 2}=e_{4, 1}$, $v_2=v_3=v_4$ and $e\neq e_{2, 3}=e_{3, 4}$. 
$v_2$ is a spike vertex since the valence of $v_2$ is more than $1$. 
By Proposition \ref{Prop b-len from shear} and Lemmas \ref{Lem gap from shear} and \ref{Lem shear from l-len and gap}, $t_1=2x_e=2l_1$, $t_2=t_3=t_4=0$ and the assigned equation follows. 
\par
Second, suppose that $v_2$ is a  puncture vertex. 
By the assumption, the valence of $v_2$ is $1$, then $e_{1, 2}=e_{2, 3}$ and $v_1=v_3$. 
$v_1$ is a spike vertex since the valence of $v_1$ is more than $1$. 
By Proposition \ref{Prop b-len from shear} and Lemmas \ref{Lem gap from shear} and \ref{Lem shear from l-len and gap}, $t_2=x_{1, 2}=l_2$ where $x_{1, 2}$ is the shear parameter of $e_{1, 2}$, $t_1=t_3=t_4=0$ and the assigned equation follows. 
\par
Finally, suppose that $v_1, v_2, v_3, v_4$ are spike vertices. 
By Lemmas \ref{Lem gap from shear} and \ref{Lem shear from l-len and gap}, $t_1=t_2=t_3=t_4=0$ and the assigned equation follows. 
\end{proof}
Let $v$ be a vertex. 
We consider the calculation of the decoration parameters from the $\lambda$-lengths. 
Take a triangle $T$ which has $v$ as a vertex, let $v_1=v, v_2, v_3$ denote the vertices of $T$, and $e_{i, j}$ denote the boundaries of $T$ which have $v_i$ and $v_j$ as the endpoints. 
Let $t_1, t_2, t_3$ denote the gaps between the horocyclic arcs and the lifts of the decoration curves along $\widetilde{e_{3, 1}}, \widetilde{e_{1, 2}}, \widetilde{e_{2, 3}}$, respectively. 
\begin{prop}\label{Prop deco from shear and l-len}(Decoration parameters from shear parameters and $\lambda$-lengths)
The decoration parameter of $v$ is 
$$d_v=\frac{1}{2}(a_{1, 2}-a_{2, 3}+a_{3, 1}-t_2+t_3-x_{e_{3, 1}}-\log\frac{\sum_{i=1}^n\exp(\sum_{j=1}^{i-1}x_{j+1})\cdot\sum_{i=1}^n\exp(\sum_{j=1}^{i-1}x_{j+2})}{\mathrm{min}_{1\le k\le n}\{\sum_{i=1}^n\exp(\sum_{j=1}^{i-1}x_{k+j})\}^2}), $$
where $x_1, \cdots, x_n$ are the shear parameters of $e_{1, 2}=e_1, e_2, \cdots, e_n$ as in the beginning of \$ 4.1. 
\end{prop}
\begin{proof}
Use the equations of $a_{1, 2}, a_{2, 3}, a_{3, 1}$ in Proposition \ref{Prop l-len from shear and deco}, eliminate the terms of the decoration parameters $d_{v_2}, d_{v_3}$ of $v_2, v_3$, and the equation of $d_{v}$ is obtained. 
Then, the terms 
$$-x_{e_{1, 2}}-\log\frac{\sum_{i=1}^n\exp(\sum_{j=1}^{i-1}x_{2, j+2})}{\sum_{i=1}^n\exp(\sum_{j=1}^{i-1}x_{2, j+1})}$$
are deformed to $-t_2$, the terms 
$$-x_{e_{2, 3}}-\log\frac{\sum_{i=1}^n\exp(\sum_{j=1}^{i-1}x_{3, j+2})}{\sum_{i=1}^n\exp(\sum_{j=1}^{i-1}x_{3, j+1})}$$
are deformed to $-t_3$, and the assigned equation follows. 
\end{proof}
We can confirm the independence of $d_v$ from the choice of the triangle $T$. 
Let $T'$ be the triangle adjacent to $T$ along $e_{3, 1}$, $v_3, v_4, v_1$ be the vertices of $T'$, $t'_i$ be the gaps for $T'$, and $d'_v$ be the assigned formula in Proposition \ref{Prop deco from shear and l-len} for $T'$. 
By Lemmas \ref{Lem gap from shear} and \ref{Lem shear from l-len and gap}, 
$$d_v-d'_v=-x_{e_{3, 1}}+\frac{1}{2}(a_{1, 2}-a_{2, 3}+a_{3, 4}-a_{4, 1}+t_1+t'_1-t_2+t_3+t'_3-t'_4)=0, $$
where $t_k+t'_k$ are the gaps of $v_k$ for odd numbers $k$ in Lemma \ref{Lem shear from l-len and gap}. 
\par
Proposition \ref{Prop deco from shear and l-len} implies that we can calculate the decoration parameters from the $\lambda$-lengths and the signed boundary lengths if the shear parameters are calculated. 
If $S$ has a spike, there exists a triangulation of $S$ such that any vertex of $V^i$ is 1-valent. 
Actually, take a spike vertex $v_0$, join the puncture vertices $v$ to $v_0$ by edges $e_v$, surround $v$ and $e_v$ with the edges $e'_v$ which have $v_0$ as both of the endpoints, triangulate the remaining area, and such a triangulation is obtained. 
\par
\proof{(Theorem \ref{Thm l-b-len coord})}
By Propositions \ref{Prop b-len from shear} and \ref{Prop l-len from shear and deco}, the map $\psi$ is continuous. 
Moreover, $\psi$ has the continuous inverse by Proposition \ref{Prop shear from l-b-len} and \ref{Prop deco from shear and l-len}. 
Since the map $\psi_x$ is a homeomorphism, the map $\psi_a=\psi\circ\psi_x$ is a homeomorphism. 
\endproof

\subsection{Examples}
\noindent
The assertion of Theorem \ref{Thm l-b-len coord} also holds true when the assumption is that any puncture vertex is 2-valent and any spike vertex is 3-valent, because the equation of the $\lambda$-length for several shear variables can be deformed to the equation for one shear variable. 
Actually, there are only two topological types of the reference surface, which has a triangulation satisfying this assumption and no triangulation satisfying the assumption of Theorem \ref{Thm l-b-len coord}. 
In this section, we consider the coordinate transformation between the shear-decoration coordinates and the $\lambda$-boundary-length coordinates of these two surfaces. 
\begin{figure}[htbp]
\begin{center}
\includegraphics[width=15cm]{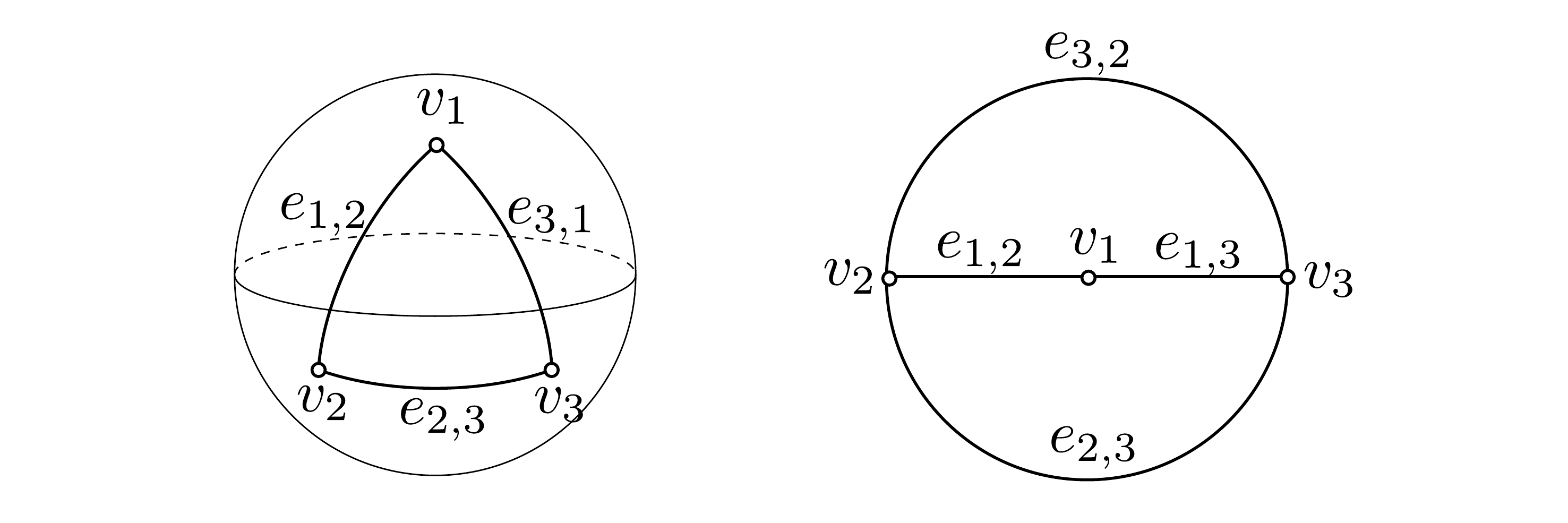}
\caption{Three-punctured sphere (left) and once punctured bigon (right). }
\label{Fig 2-val ex}
\end{center}
\end{figure}
\begin{ex}\label{Ex 3-punc sphere}(Three-punctured sphere)
Let S be the three-punctured sphere. 
For $i, j\in\{1, 2, 3\}$, let $v_i$ and $e_{i, j}$ denote the vertices and the edges as the left of Figure \ref{Fig 2-val ex}. 
Let $x_{i, j}$ denote the shear parameters of $e_{i, j}$, $d_i$ denote the decoration parameters of $v_i$, $a_{i, j}$ denote the $\lambda$-lengths of $e_{i, j}$, and $l_i$ denote the signed boundary lengths of $v_i$. 
By Propositions \ref{Prop b-len from shear} and \ref{Prop l-len from shear and deco}, 
\begin{eqnarray*}
a_{i, j}\!\!\!&=\!\!\!&x_{i, j}+\sum_{l\in\{i, j\}}(d_l+\log\frac{1+e^{x_{k, l}}}{1+\mathrm{min}\{e^{x_{i, j}}, e^{x_{k, l}}\}}), \\
l_i\!\!\!&=\!\!\!&x_{i, j}+x_{i, k}, 
\end{eqnarray*}
for $i, j, k$ with $\{i, j, k\}=\{1, 2, 3\}$. 
Because of the valence of the vertices, we can solve the equation for the shear parameters and the decoration parameters: 
\begin{eqnarray*}
x_{i, j}\!\!\!&=\!\!\!&\frac{1}{2}(l_i+l_j-l_k), \\
d_i\!\!\!&=\!\!\!&\frac{1}{2}(a_{i, j}+a_{i, k}-a_{j, k}-l_i)-\log2\cosh\frac{1}{4}(l_i-l_j-l_k)\\
&&+\log(1+\mathrm{min}\{e^{\frac{1}{2}(l_i+l_j-l_k)}, e^{\frac{1}{2}(l_i+l_k-l_j)}\}), 
\end{eqnarray*}
for $i, j, k$ with $\{i, j, k\}=\{1, 2, 3\}$. 
\end{ex}
\begin{ex}\label{Ex 1-punc bigon}(Once punctured bigon)
Let S be the once punctured bigon. 
For $i, j\in\{1, 2, 3\}$, let $v_i$ and $e_{i, j}$ denote the vertices and the edges as the right of Figure \ref{Fig 2-val ex}. 
Set the parameters as Example \ref{Ex 3-punc sphere}. 
By Propositions \ref{Prop b-len from shear} and \ref{Prop l-len from shear and deco}, 
\begin{eqnarray*}
a_{1, i}\!\!\!&=\!\!\!&d_1+d_i+\log\frac{1+e^{x_{1, j}}}{1+\mathrm{min}\{e^{x_{1, i}}, e^{x_{1, j}}\}}+\log\frac{1+e^{x_{1, i}}}{1+\mathrm{min}\{e^{x_{1, i}}, e^{-x_{1, i}}\}}, \\
a_{i, j}\!\!\!&=\!\!\!&d_i+d_j+\log\frac{1+e^{x_{1, i}}}{1+\mathrm{min}\{e^{x_{1, i}}, e^{-x_{1, i}}\}}+\log\frac{1+e^{-x_{1, j}}}{1+\mathrm{min}\{e^{x_{1, j}}, e^{-x_{1, j}}\}}, \\
l_1\!\!\!&=\!\!\!&x_{1, 2}+x_{1, 3}, 
\end{eqnarray*}
for $i, j$ with $\{i, j\}=\{2, 3\}$. 
Because of the valence of the vertices, we can solve the equation for the shear parameters and the decoration parameters: 
\begin{eqnarray*}
x_{1, i}\!\!\!&=\!\!\!&\frac{1}{2}(l_1+a_{i, j}-a_{j, i}), \\
d_1\!\!\!&=\!\!\!&\frac{1}{2}(a_{1, 2}+a_{1, 3}-\log(e^{\frac{1}{2}(l_1+a_{i, j}+a_{j, i})}+\mathrm{max}\{e^{l_1+a_{2, 3}}, e^{l_1+a_{3, 2}}\}), \\
d_i\!\!\!&=\!\!\!&\frac{1}{2}(a_{1, i}-a_{1, j}+a_{i, j}+\log(1+\mathrm{min}\{e^{\frac{1}{2}(l_1+a_{i, j}-a_{j, i})}, e^{-\frac{1}{2}(l_1+a_{i, j}-a_{j, i})}\})\\
&&-\log(1+e^{\frac{1}{2}(l_1+a_{i, j}-a_{j, i})})(1+e^{-\frac{1}{2}(l_1+a_{j, i}-a_{i, j})})), 
\end{eqnarray*}
for $i, j$ with $\{i, j\}=\{2, 3\}$. 
\end{ex}

\section{Lamination spaces}
\noindent
There are spaces of laminations corresponding to the generalized Teichm\"{u}ller spaces. 
In this chapter, we recall these spaces of laminations, and introduce the generalized spaces of laminations corresponding to the decorated enhanced Teichm\"{u}ller spaces. 
Then, we find the common properties to the specific curves of laminations. 
\subsection{Teichm\"{u}ller spaces and lamination spaces}
\noindent
$\mathcal{X}$-laminations (or unbounded measured laminations) and $\mathcal{A}$-laminations (or bounded measured laminations) are defined by Fock and Goncharov in \cite{Fock-Goncharov}. 
The former are corresponding to the enhanced Teichm\"{u}ller spaces, and the latter are corresponding to the decorated Teichm\"{u}ller spaces. 
For the purpose of the definition of $\mathcal{AX}$-laminations, We define these known laminations in other words. 
\par
In this chapter, simple closed curves on $S$ and simple arcs on $S$ whose endpoints are in $V$ or the boundaries of $S$ are called {\it curves} on $S$. 
\begin{itemize}
\item A curve $\gamma$ is {\it contractible} if it satisfies one of the following: 
\begin{description}
\item[($C_1$-$1$)] there exists a boundary edge $e\in E^b(\Gamma)$ such that the endpoints of $\gamma$ are in the interior of $e$, and there is a subarc $\delta$ of $e$ such that $\gamma$ and $\delta$ bound a disk on $S$, 
\item[($C_1$-$2$)] $\gamma$ bounds a disk on $S$, 
\item[($C_2$-$1$)] there exist a spike vertex $v\in V^b$ and a boundary edge $e\in E^b(\Gamma)$ which has $v$ as an endpoint such that an endpoint of $\gamma$ is $v$ and the other is in the interior of $e$, and there is a subarc $\delta$ of $e$ such that $\gamma$, $\delta$ and $v$ bound a disk on $S$, 
\item[($C_2$-$2$)] there exists a vertex $v\in V$ such that the endpoints of $\gamma$ are $v$ and that $\gamma$ and $v$ bound a disk on $S$. 
\end{description}
\item A curve $\gamma$ is called an {\it $\mathcal{A}$-curve} if $\gamma$ is not contractible and satisfies either: 
\begin{description}
\item[($A$-$1$)] there exist boundary edges $e^+, e^-\in E^b(\Gamma)$ such that the endpoints of $\gamma$ are in the interior of $e^{\pm}$, and there are subarcs $\delta^{\pm}$ of $e^{\pm}$ such that $\gamma$ and $\delta^{\pm}$ bound a disk on $S$. 
\item[($A$-$2$)] there exists a puncture vertex $v\in V^i$ such that $\gamma$ and $v$ bound an annulus on $S$. 
\end{description}
\item A curve $\gamma$ is called an {\it $\mathcal{X}_D$-curve} if $\gamma$ is not contractible and 
\begin{description}
\item[($X$-$1$)] there exists a spike vertex $v\in V^b$ such that an endpoint of $\gamma$ is $v$. 
\end{description}
\item A curve $\gamma$ is called an {\it $\mathcal{X}$-curve} if $\gamma$ is neither contractible nor $\mathcal{X}_D$-curve and 
\begin{description}
\item[($X$-$2$)] there exists a puncture vertex $v\in V^i$ such that an endpoint of $\gamma$ is $v$. 
\end{description}
\item Otherwise, a curve $\gamma$ is {\it general}. 
\end{itemize}
\begin{figure}[htbp]
\begin{center}
\includegraphics[width=15cm]{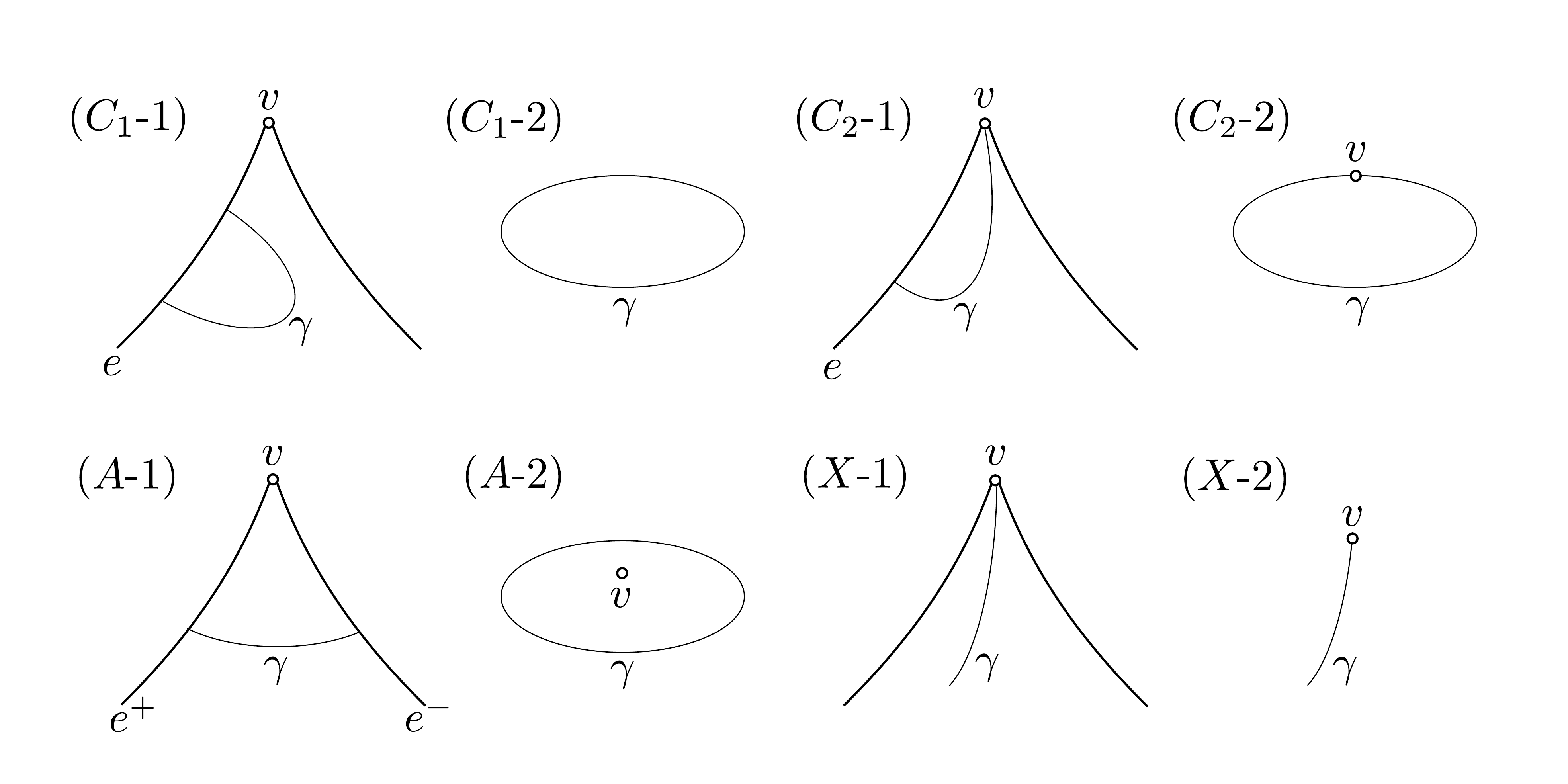}
\caption{features of curves}
\label{Fig curves}
\end{center}
\end{figure}
Consider the doubled surface $S_D$ of $S$, Conditions ($C_1$-$1$), ($C_2$-$1$), ($A$-$1$) and ($X$-$1$) are regard as Conditions ($C_1$-$2$), ($C_2$-$2$), ($A$-$2$) and ($X$-$2$), respectively, for the arcwise connected components of the doubled curve $\gamma_D$ of $\gamma$. 
\begin{dfn}\label{Def lami}(Laminations)
A {\it rational measured lamination} on $S$ is a finite collection of curves on $S$ with rational weights, and satisfies the following conditions: 
\begin{itemize}
\item A curve of negative weight is $\mathcal{A}$-curve. 
\item If two curves are intersecting, one of these curves is an $\mathcal{A}$-curve and the other is an $\mathcal{X}$-curve or an $\mathcal{X}_D$-curve. 
\end{itemize}
Moreover, it is subject to the following equivalence relations: 
\begin{itemize}
\item A lamination containing a contractible curve or a curve of weight $0$ is equivalent to the lamination with these curves removed. 
\item A lamination containing two homotopy (rel endpoints) equivalent curves $\gamma_1$ and $\gamma_2$ of weight $w_1$ and $w_2$, respectively, is equivalent to the lamination with the weight $w_1+w_2$ on $\gamma_1$ and with $\gamma_2$ removed. 
\end{itemize}
\end{dfn}
\begin{dfn}\label{Def A-lami}($\mathcal{A}$-laminations)
A {\it rational $\mathcal{A}$-lamination} (or a {\it rational bounded measured lamination}) is a rational measured lamination which has no $\mathcal{X}$-curves and no $\mathcal{X}_D$-curves. 
The set of rational $\mathcal{A}$-laminations on $S$ is denoted by $T^a(S, \mathbb{Q})$. 
\end{dfn}
For an $\mathcal{A}$-lamination $L$ and an edge $e$, let $a_e(L)$ denote the sum of the weights of curves intersected with $e$, where the representative of $L$ is minimal intersected with edges. 
\begin{prop}\label{Prop coord of A-lami}\cite{Fock-Goncharov}(Coordinates of $\mathcal{A}$-lamination space)
The map 
$$\Phi_{\mathbb{Q}, a} : T^a(S, \mathbb{Q})\ni L\mapsto(a_e(L))_{e\in E}\in\mathbb{Q}^E$$
from rational $\mathcal{A}$-laminations to rational weights of edges is a bijection. 
\end{prop}
The topology of $T^a(S, \mathbb{Q})$ is obtained by $\mathbb{Q}^E$, and we can define the space of {\it real $\mathcal{A}$-laminations} $T^a(S, \mathbb{R})$ as the completion of $T^a(S, \mathbb{Q})$. 
Let $\Phi_a$ denote the correspondence from $T^a(S, \mathbb{R})$ to $\mathbb{R}^E$. 
Proposition \ref{Prop coord of A-lami} implies that the space of real $\mathcal{A}$-laminations is parametrized by the product of $\mathbb{R}$ as the decorated Teichm\"{u}ller space. 
\begin{dfn}\label{Def ori map}(Orientation maps)
For a rational measured lamination $L$, a map $\epsilon:V\rightarrow\{0, \pm1\}$ is called an {\it orientation map} if $\epsilon^{-1}(\{\pm1\})=V_L$, where $V_L$ is the set of vertices that are endpoints of curves of non-zero weight. 
\end{dfn}
\begin{dfn}\label{Def X-lami}($\mathcal{X}$-laminations)
A {\it rational $\mathcal{X}$-lamination} (or a {\it rational unbounded measured lamination}) $(L, \epsilon)$ is a pair of a rational measured lamination $L$ which has no $\mathcal{A}$-curves and no $\mathcal{X}_D$-curves and an orientation map $\epsilon$. 
The set of rational $\mathcal{X}$-laminations on $S$ is denoted by $T^x(S, \mathbb{Q})$. 
\end{dfn}
Let $(L, \epsilon)$ be an $\mathcal{X}$-lamination, $e$ be an interior edge, and edges $e_{1, 2}, e_{2, 3}, e_{3, 4}, e_{4, 1}$ be the edges as the right of Figure \ref{Fig X-lami}. 
Take the representative of $L$ spiraling at endpoints infinitely as shown in the left of Figure \ref{Fig X-lami}. 
A curve intersects with $e$ {\it positively} if the intersection is interposed between the intersections with $e_{1, 2}$ and $e_{3, 4}$ as the curve $\gamma$ in Figure \ref{Fig X-lami}. 
A curve intersects with $e$ {\it negatively} if the intersection is interposed between the intersections with $e_{2, 3}$ and $e_{4, 1}$ as the curve $\gamma'$ in Figure \ref{Fig X-lami}. 
Let $x_e(L, \epsilon)$ denote the sum of the weights of curves intersected with $e$ positively minus the sum of the weights of curves intersected with $e$ negatively. 
Other curves as the curve $\delta$ and $\delta'$ in Figure \ref{Fig X-lami} do not effect on $x_e(L, \epsilon)$. 
\begin{figure}[htbp]
\begin{center}
\includegraphics[width=15cm]{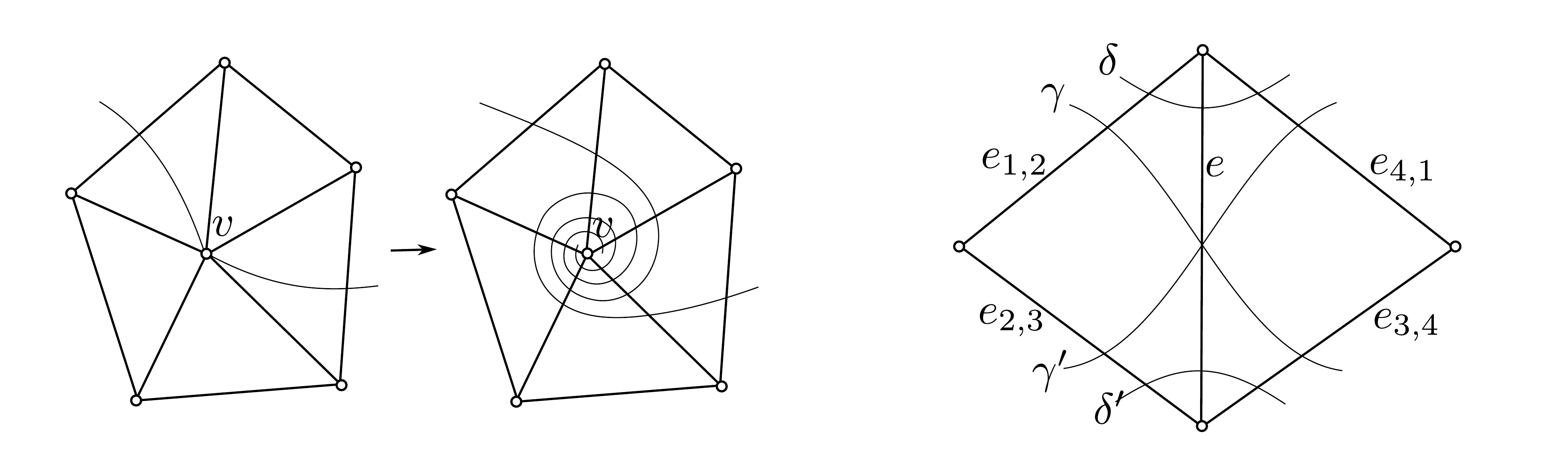}
\caption{Spiraling representative in the case that $\epsilon(v)=1$ (left), and curves intersect with edges positively and negatively (right). }
\label{Fig X-lami}
\end{center}
\end{figure}
\begin{prop}\label{Prop coord of X-lami}\cite{Fock-Goncharov}(Coordinates of $\mathcal{X}$-lamination space)
The map 
$$\Phi_{\mathbb{Q}, x} : T^x(S, \mathbb{Q})\ni (L, \epsilon)\mapsto(x_e(L, \epsilon))_{e\in E^i}\in\mathbb{Q}^{E^i}$$
from rational $\mathcal{X}$-laminations to rational weights of internal edges is a bijection. 
\end{prop}
The topology of $T^x(S, \mathbb{Q})$ is obtained by $\mathbb{Q}^{E^i}$, and we can define the space of {\it real $\mathcal{X}$-laminations} $T^x(S, \mathbb{R})$ as the completion of $T^x(S, \mathbb{Q})$. 
Let $\Phi_x$ denote the correspondence from $T^x(S, \mathbb{R})$ to $\mathbb{R}^{E^i}$. 
Proposition \ref{Prop coord of X-lami} implies that the space of real $\mathcal{X}$-laminations is parametrized by the product of $\mathbb{R}$ as the enhanced Teichm\"{u}ller space. 

\subsection{$\mathcal{AX}$-lamination spaces}
\noindent
In this section, we introduce the generalized spaces of laminations corresponding to the decorated enhanced Teichm\"{u}ller spaces. 
Then, we research the meaning of curves of each type. 
\begin{dfn}\label{Def AX-lami}($\mathcal{AX}$-laminations)
A {\it rational $\mathcal{AX}$-lamination} $(L, \epsilon)$ is a pair of a rational measured lamination $L$ which has no $\mathcal{X}_D$-curves and an orientation map $\epsilon$. 
The set of rational $\mathcal{AX}$-laminations on $S$ is denoted by $T^{ax}(S, \mathbb{Q})$. 
\end{dfn}
Let $(L, \epsilon)$ be a rational $\mathcal{AX}$-lamination and $L_a$ be the lamination composed of $\mathcal{A}$-curves of $L$. 
Let $d_v(L_a)$ denote the weight of the $\mathcal{A}$-curve corresponding to $v$, where the correspondence is shown in Figure \ref{Fig curves}. 
Since $L\setminus L_a$ has no $\mathcal{A}$-curves, $x_e(L\setminus L_a, \epsilon)$ can be calculated as rational $\mathcal{X}$-laminations. 
By Proposition \ref{Prop coord of X-lami}, the following proposition follows. 
\begin{prop}\label{Prop coord of AX-lami}(Coordinates of $\mathcal{AX}$-lamination space)
This map 
$$\Psi_{\mathbb{Q}, x} : T^{ax}(S, \mathbb{Q})\ni (L, \epsilon)\mapsto((x_e(L\setminus L_a, \epsilon))_{e\in E^i}, (d_v(L_a))_{v\in V})\in\mathbb{Q}^{E^i}\times\mathbb{Q}^{V}$$
from rational $\mathcal{AX}$-laminations to rational weights of internal edges and rational weights of $\mathcal{A}$-curves is a bijection. 
\end{prop}
The topology of $T^{ax}(S, \mathbb{Q})$ is obtained by $\mathbb{Q}^{E^i}\times\mathbb{Q}^{V}$, and we can define the space of {\it real $\mathcal{AX}$-laminations} $T^{ax}(S, \mathbb{R})$ as the completion of $T^{ax}(S, \mathbb{Q})$. 
Let $\Psi_x$ denote the correspondence from $T^{ax}(S, \mathbb{R})$ to $\mathbb{R}^{E^i}\times\mathbb{R}^{V}$. 
Proposition \ref{Prop coord of AX-lami} implies that the space of real $\mathcal{AX}$-laminations is parametrized by the product of $\mathbb{R}$ as the decorated enhanced Teichm\"{u}ller space, and weights of $\mathcal{A}$-curves correspond to decoration parameters. 
\par
For a boundary edge $e$, we can also define $x_e$. 
Consider the doubled surface $S_D$ of $S$, the doubled lamination $L_D$ of $L$ and the doubled orientation $\epsilon_D$ of $\epsilon$ such that these are invariant under the involution. 
$x_e(L, \epsilon)$ of the boundary edge $e$ of $S$ is defined by $x_e(L_D, \epsilon_D)$ of the interior edge $e$ of $S_D$. 
This definition corresponds to considering the shear parameters of boundary edges. 
\begin{dfn}\label{Def XD-lami and AXD-lami}($\mathcal{X}_D$-laminations and $\mathcal{AX}_D$-laminations)
A {\it rational $\mathcal{X}_D$-lamination} $(L, \epsilon)$ is a pair of a rational measured lamination $L$ which has no $\mathcal{A}$-curves and an orientation map $\epsilon$. 
The set of rational $\mathcal{X}_D$-laminations on $S$ is denoted by $T^x_D(S, \mathbb{Q})$. 
A {\it rational $\mathcal{AX}_D$-lamination} $(L, \epsilon)$ is a pair of a rational measured lamination $L$ and an orientation map $\epsilon$. 
The set of rational $\mathcal{AX}_D$-laminations on $S$ is denoted by $T^{ax}_D(S, \mathbb{Q})$. 
\end{dfn}
By the argument similar to Proposition \ref{Prop coord of X-lami} and \ref{Prop coord of AX-lami}, the maps 
$$\Phi_{\mathbb{Q}, x, D} : T^x_D(S, \mathbb{Q})\ni (L, \epsilon)\mapsto(x_e(L, \epsilon))_{e\in E}\in\mathbb{Q}^{E}\mathrm{\ and}$$
$$\Psi_{\mathbb{Q}, x, D} : T^{ax}_D(S, \mathbb{Q})\ni (L, \epsilon)\mapsto((x_e(L\setminus L_a, \epsilon))_{e\in E}, (d_v(L_a))_{v\in V})\in\mathbb{Q}^E\times\mathbb{Q}^{V}$$
are bijective, and these completions 
$$\Phi_{x, D} : T^x_D(S, \mathbb{R})\rightarrow\mathbb{R}^{E}\mathrm{\ and\ }\Psi_{x, D} : T^{ax}_D(S, \mathbb{R})\rightarrow\mathbb{R}^E\times\mathbb{R}^{V}$$
are defined. 
\par
\begin{prop}\label{Prop ori and enha}(Orientations and enhancements)
Let $[(L_n, \epsilon_n)_{n\in\mathbb{N}}]$ be a real $\mathcal{AX}$-lamination, let $[X, f, \varepsilon, D]$ be a decorated enhanced marked hyperbolic surface, and let $v$ be an interior vertex. 
If $[(L_n, \epsilon_n)_{n\in\mathbb{N}}]$ and $[X, f, \varepsilon, D]$ are parametrized to the same point of $\mathbb{R}^{E^i}\times\mathbb{R}^{V}$ and $\epsilon_n(v)$ converges to $0$, $\varepsilon(C_v)$ equals to $0$, where $C_v$ is the cusp or the closed geodesic boundary corresponding to $v$. 
\end{prop}
\begin{proof}
Since $\epsilon_n(v)$ converges to $0$, for sufficiently large $n$, $\epsilon_n(v)$ equals to $0$. 
Let $N_v$ be the union of triangles which have $v$ as a vertex. 
Take a component $\gamma$ of the intersection of $N_v$ and curves of $L_n$ whose weight $w$ is not $0$. 
By Proposition \ref{Prop b-len from shear}, $\varepsilon(C_v)$ is the sign of the signed boundary length $l_v$ that is the sum of shear parameters of edges which have $v$ as an endpoint. 
Consider the effect of $\gamma$ on $l_v$. 
$\gamma$ intersects at most once with edges which have $v$ as an endpoint, and $\gamma$ does not have $v$ since $\epsilon_n(v)$ is $0$. 
When one goes along $\gamma$ clockwise direction, $\gamma$ adds $w$ at the first intersection and subtracts $w$ at the last intersection. 
In the total result, $\gamma$ has no effect on $l_v$. 
Therefore, $\varepsilon(C_v)$ equals to $0$. 
\end{proof}
Proposition \ref{Prop ori and enha} implies that orientations and enhancements are corresponding. 
However the converse proposition of Proposition \ref{Prop ori and enha} is not applied. 

\subsection{Compatibility of parametrization}
\noindent
Weighted general curves are corresponding to the shear coordinates of a marked hyperbolic surface by the map $\Phi_x$. 
On the other hand, they are corresponding to the $\lambda$-length coordinates of a marked hyperbolic surface by the map $\Phi_a$. 
Following theorem states that these marked hyperbolic surfaces are same. 
Although $T^x(S, \mathbb{R})$ and $T^a(S, \mathbb{R})$ are defined independently, their common curves have the same meaning. 
\par
Let
$$\Psi_a : T^{ax}(S, \mathbb{R})\rightarrow\mathbb{R}^E\times\mathbb{R}^{V^i}$$
be the composition of $\Psi_x$ and $\psi$. 
\par
\begin{thm}\label{Thm compati of para}(Compatibility of parametrization)
The equations 
$$\Psi_x\circ I_x=\iota_x\circ\Phi_x\mathrm{\ and\ }\Psi_a\circ I_a=\iota_a\circ\Phi_a$$
hold, where $\iota_x$ is the inclusion map from $\mathbb{R}^{E^i}$ to $\mathbb{R}^{E^i}\times\mathbb{R}^V$, $\iota_a$ is the inclusion map from $\mathbb{R}^E$ to $\mathbb{R}^E\times\mathbb{R}^{V^i}$, $I_x$ is the inclusion map from $T^x(S, \mathbb{R})$ to $T^{ax}(S, \mathbb{R})$, and $I_a$ is the inclusion map from $T^a(S, \mathbb{R})$ to $T^{ax}(S, \mathbb{R})$ as in Figure \ref{Fig compati of para 1}. 
\end{thm}
\begin{figure}[htbp]
\begin{center}
\includegraphics[width=12cm]{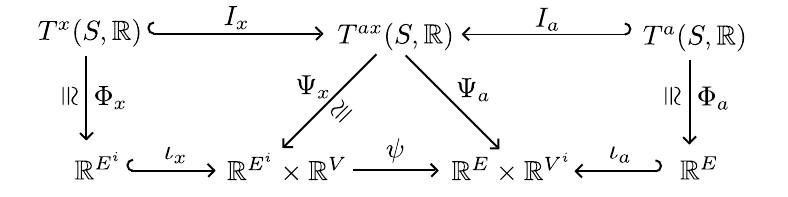}
\caption{The commutative diagram describing that general curves have the same meaning in $T^x(S, \mathbb{R})$ and $T^a(S, \mathbb{R})$. }
\label{Fig compati of para 1}
\end{center}
\end{figure}
Moreover, Theorem \ref{Thm compati of para} implies that the weight of the $\mathcal{A}$-curve of a lamination $L$ around a vertex $v$ is the decoration parameter $d_v$ of the decorated marked hyperbolic surface whose $\lambda$-length coordinates are $\Phi_a(L)$. 
In the remains of this section, we prove this theorem. 
\par
By the construction of $\Phi_x$ and $\Psi_x$, it is immediately follows the former commutative equation. 
Since the maps in the statement are continuous, it is sufficient to show that 
$$\psi\circ\Psi_x\circ I_a(L)=\iota_a\circ\Phi_a(L)$$ 
for any rational $\mathcal{A}$-lamination $L$. 
\par
Let $e$ be an edge, and let $v, v'$ be the endpoints of $e$. 
Lift the edges which have $v$ as an endpoint to the upper half plane $\mathbb{H}^2$ such that a lift $\tilde{v}$ of $v$ is $\infty$. 
Suppose that there is no curve of $L$ intersecting with $e$ negatively. 
(Otherwise, consider the mirror image.)
Lift edges which have $v$ as an endpoint to the upper half plane $\mathbb{H}^2$ such that a lift $\tilde{v}$ of $v$ is $\infty$. 
Let $\tilde{\gamma}_1, \cdots , \tilde{\gamma}_n$ denote the lifts of curves of $L$ intersecting with $e$ positively, $\tilde{\delta}, \tilde{\delta}'$ denote the lifts of $\mathcal{A}$-curves of $L$ around $v, v'$ respectively, and $\tilde{\delta}_1, \cdots , \tilde{\delta}_m, \tilde{\delta}'_1, \cdots , \tilde{\delta}'_{m'}$ denote the lifts of remain general curves of $L$ intersecting with $e$ as shown in Figure \ref{Fig compati of para 2}. 
Let $\tilde{e}$ denote the lift of $e$, and $\tilde{e}_k$ denote the lift of the edge which has $v$ as an endpoint and intersects with $\delta_k$ positively as shown in Figure \ref{Fig compati of para 2}. 
It may occur that $\tilde{e}_i=\tilde{e}_j$ for distinct $i, j$. 
Take the base point $b$ of $\tilde{e}$ with respect to the triangle on the right side of $\tilde{e}$, the base point $b_1$ of $\tilde{e}_1$ with respect to the triangle on the right side of $\tilde{e}_1$, and the lift $\tilde{\alpha}_1$ of the horocycle $\alpha_1$ around $v$ passing through $b_1$. 
\begin{figure}[htbp]
\begin{center}
\includegraphics[width=12cm]{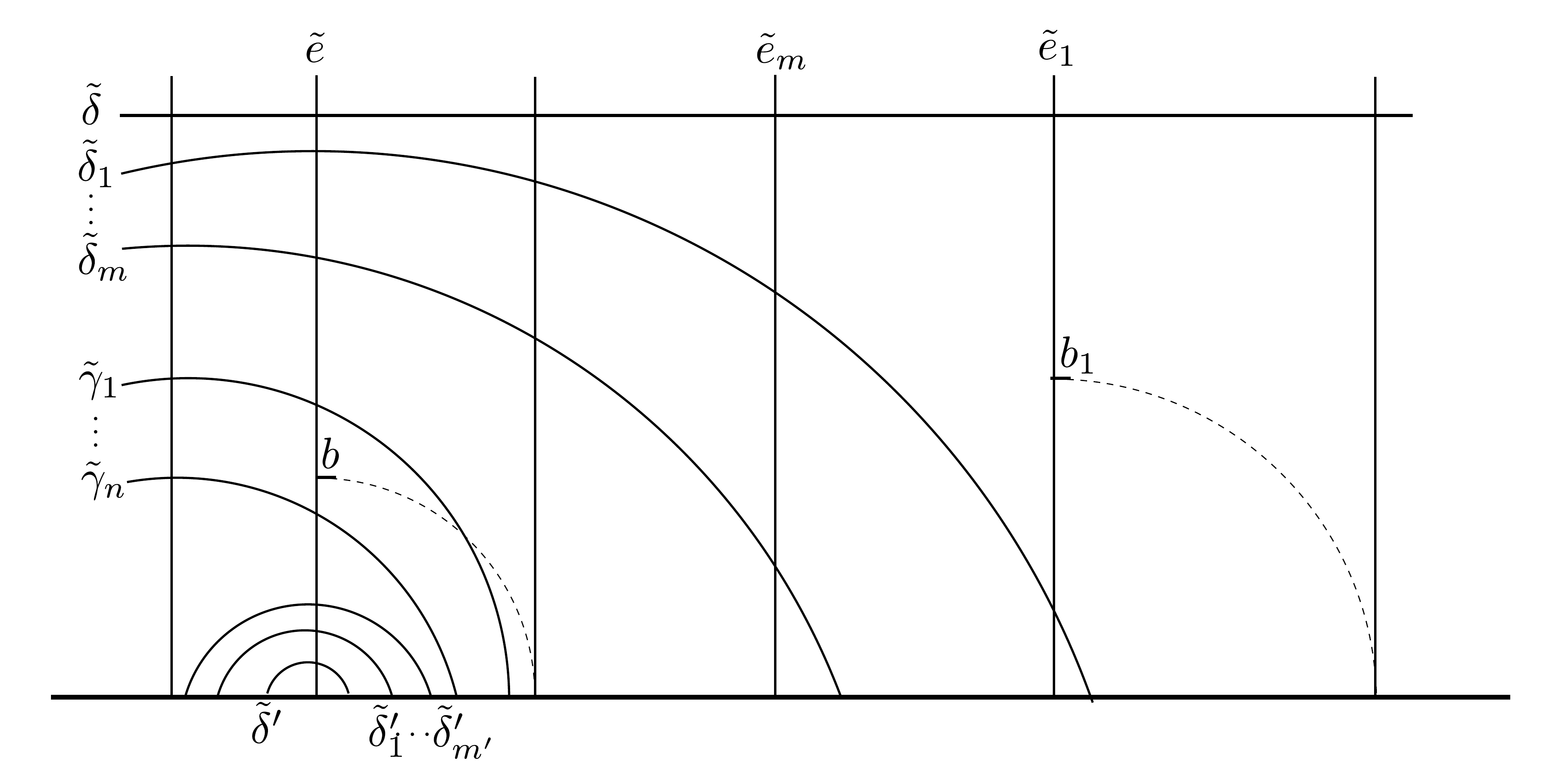}
\caption{The labels of curves of $L$. (In the case that there is a curve of $L$ intersecting with $e$ negatively, consider the mirror image. )}
\label{Fig compati of para 2}
\end{center}
\end{figure}
\par
\begin{lem}\label{Lem deco pass b1}(Decoration curves passing through $b_1$)
$\alpha_1$ is the decoration curve around $v$ when the shear coordinates of the decorated marked hyperbolic surface are $\Psi_x\circ I_a(L)$ and decoration parameter of $v$ is $0$. 
\end{lem}
\begin{proof}
If there exist a curve $\gamma$ of $L$ and an edge $e_-$ such that they are intersecting negatively and the real part of $\tilde{e}_-$ is lower than the real part of $\tilde{e}_1$, there must be an edge $e_+$ which intersects with $\gamma$ positively and the real part of $\tilde{e}_+$ is not larger than the real part of $\tilde{e}_1$. 
That is to say, the negative factor of the shear parameter of an edge $\tilde{e}_-$ on the left side of $\tilde{e}_1$ must be canceled by the positive factor of the shear parameter of $\tilde{e}_1$ or the edge $\tilde{e}_+$ on the left side of $\tilde{e}_1$. 
It may occur that the positive factor of the shear parameter of an edge on the left side of $\tilde{e}_1$ is not canceled, when $v$ is spike vertex. 
Similarly, the positive factor of the shear parameter of an edge on the right side of $\tilde{e}_1$ must be canceled by the negative factor of the shear parameter of $\tilde{e}_1$ or an edge on the right side of $\tilde{e}_1$. 
\par
\begin{figure}[htbp]
\begin{center}
\includegraphics[width=12cm]{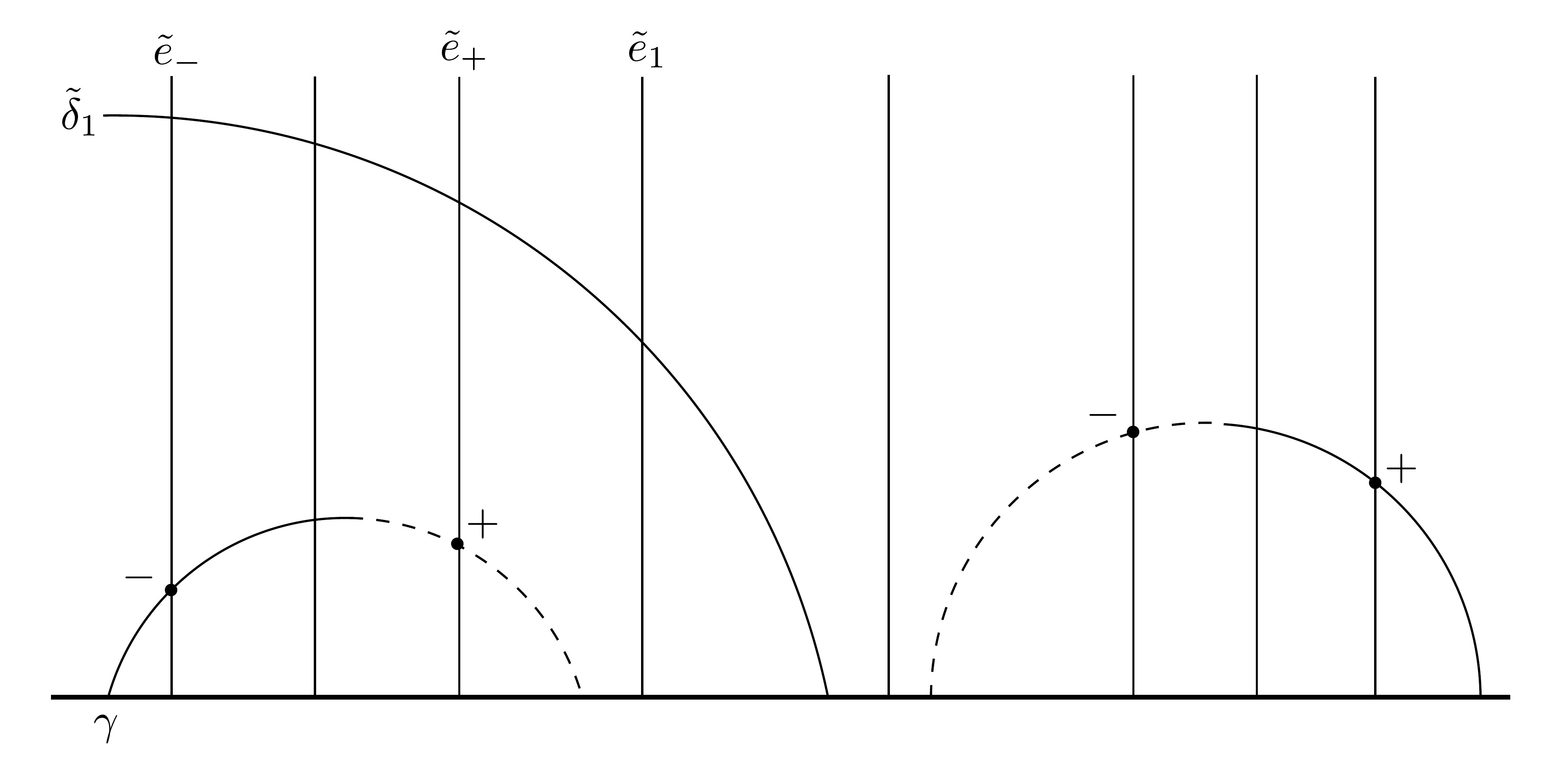}
\caption{Factors of shear parameters which decrease the imaginary part of $b_1$ are canceled. }
\label{Fig deco pass b1}
\end{center}
\end{figure}
These considerations imply that the imaginary part of $b_1$ is larger than the imaginary parts of other base points. 
By the definition of the origin of decoration parameters, the decoration curve around $v$ passes through $b_1$ when the decoration parameter of $v$ is $0$. 
\end{proof}
\begin{lem}\label{Lem len from b to alpha}(The signed length from $b$ to $\tilde{\alpha}_1$)
The signed length from $b$ to $\tilde{\alpha}_1$ is 
$$\sum_{k=1}^mw_{\delta_k}, $$
when $w_\gamma$ denotes the weight of $\gamma$ for a curve $\gamma$ of $L$. 
\end{lem}
\begin{proof}
By Lemma \ref{Lem deco pass b1}, the negative factor of the shear parameter of an edge on the left side of $\tilde{e}_1$ must be canceled by the positive factor of the shear parameter of $\tilde{e}_1$ or an edge on the left side of $\tilde{e}_1$. 
If there exist a curve $\gamma$ of $L$ other than $\delta_1, \cdots, \delta_m$ and an edge $e_+$ such that they are intersecting positively and the real part of $\tilde{e}_+$ is not larger than the real part of $\tilde{e}_k$, there must be an edge $e_-$ which intersects with $\gamma$ negatively and the real part of $\tilde{e}_-$ is larger than the real part of $\tilde{e}_{k+1}$, where $\tilde{e}_{m+1}$ is defined as $e$. 
That is to say, the positive factor of the shear parameter of $\tilde{e}_k$ or an edge $\tilde{e}_+$ on the left side of $\tilde{e}_k$ obtained from $\gamma$ must be canceled by the negative factor of the shear parameter of the edge $\tilde{e}_-$ on the right side of $\tilde{e}_{k+1}$. 
\par
\begin{figure}[htbp]
\begin{center}
\includegraphics[width=12cm]{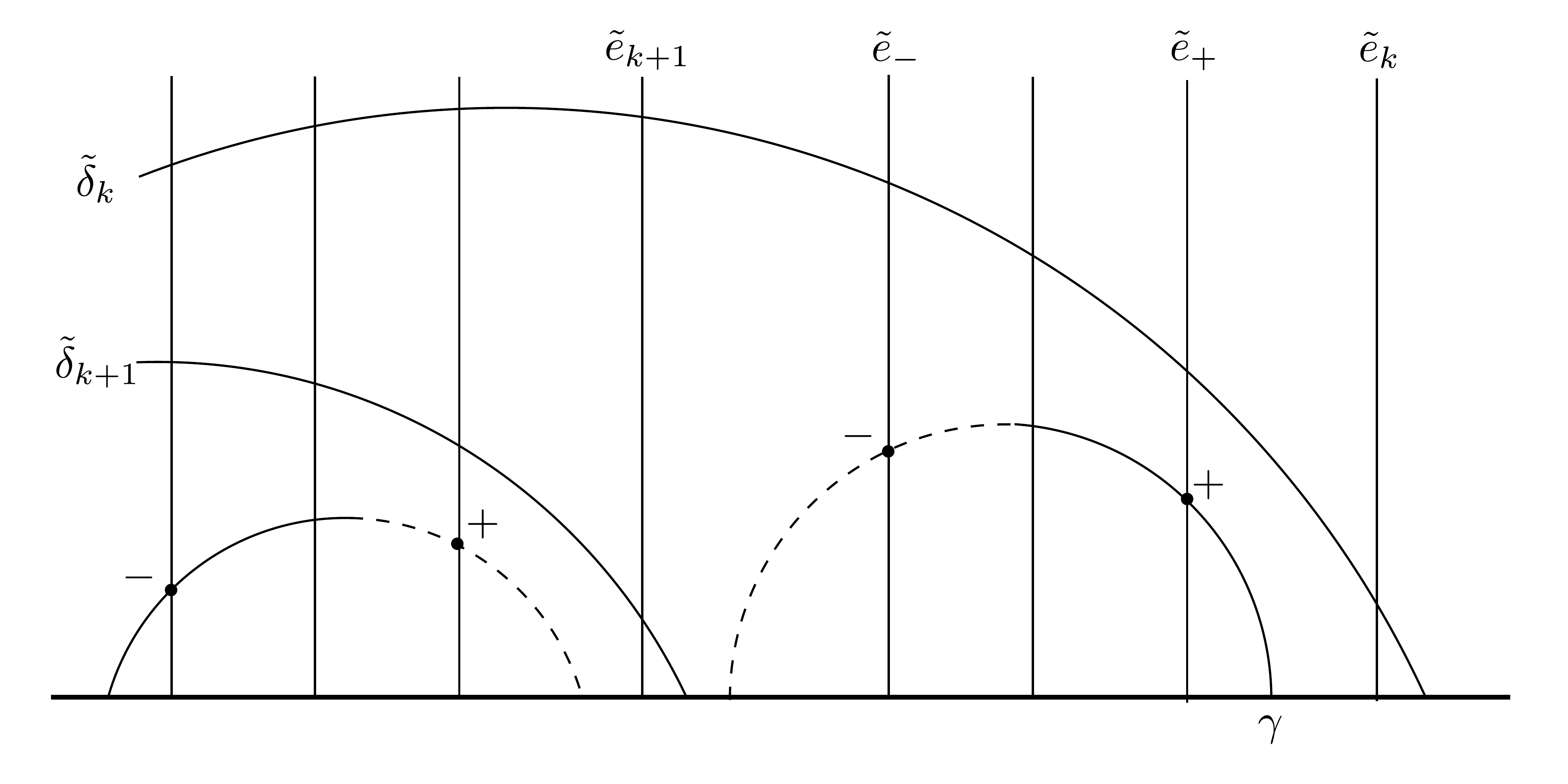}
\caption{Factors of shear parameters which are obtained from curves other than $\delta_1, \cdots, \delta_m$ are canceled. }
\label{Fig len from b to alpha}
\end{center}
\end{figure}
These considerations imply that factors of shear parameters which are not canceled are positive factors obtained from $\delta_1, \cdots, \delta_m$. 
Then, assigned statement follows. 
\end{proof}
Consider the lifting to the upper half plane $\mathbb{H}^2$ such that a lift $\tilde{v}'$ of $v'$ is $\infty$, and we have results about $v'$ similar to Lemma \ref{Lem deco pass b1} and Lemma \ref{Lem len from b to alpha}. 
\par
\proof{(Theorem \ref{Thm compati of para})}
$\psi\circ\Psi_x\circ I_a(L)$ is the tuple of the $\lambda$-lengths and the signed boundary lengths of the decorated enhanced hyperbolic surface whose shear-decoration coordinates are $\Psi_x\circ I_a(L)$. 
By Lemma \ref{Lem deco pass b1} and Lemma \ref{Lem len from b to alpha}, 
\begin{eqnarray*}
p_e\circ\psi\circ\Psi_x\circ I_a(L)\!\!\!&=\!\!\!&\pi_v\circ\Psi_x\circ I_a(L)+\sum_{k=1}^mw_{\delta_k}+\pi_{v'}\circ\Psi_x\circ I_a(L)+\sum_{k=1}^{m'}w_{\delta'_k}+\pi_e\circ\Psi_x\circ I_a(L)\\
&=\!\!\!&w_\delta+\sum_{k=1}^mw_{\delta_k}+w_{\delta'}+\sum_{k=1}^{m'}w_{\delta'_k}+\sum_{k=1}^nw_{\gamma_k}\\
&=\!\!\!&p_e\circ\iota_a\circ\Phi_a(L), 
\end{eqnarray*}
where $\pi_e$, $\pi_v$ and $\pi_{v'}$ are the projection maps from $\mathbb{R}^{E^i}\times\mathbb{R}^V$ to $\mathbb{R}$ corresponding to the subscripts, and $p_e$ is the projection map from $\mathbb{R}^E\times\mathbb{R}^{V^i}$ to $\mathbb{R}$ corresponding to the subscript. 
By Proposition \ref{Prop ori and enha}, 
\begin{eqnarray*}
p_v\circ\psi\circ\Psi_x\circ I_a(L)\!\!\!&=\!\!\!&0\\
&=\!\!\!&p_v\circ\iota_a\circ\Phi_a(L), 
\end{eqnarray*}
where $p_v$ is the projection map from $\mathbb{R}^E\times\mathbb{R}^{V^i}$ to $\mathbb{R}$ corresponding to the subscript. 
Therefore, Theorem \ref{Thm compati of para} follows. 
\endproof

\section*{Acknowledgements}
\noindent
The author would like to express the deepest appreciation to Hideki Miyachi for educational guiding from a basic level. 
The author would also like to express his gratitude to Ken'ichi Ohshika and Shinpei Baba for their helpful advices and encouragements.

\end{document}